\documentclass[12pt,oneside]{amsart}

\usepackage{amsmath}

\calclayout

\usepackage{amsfonts,amssymb,amsthm,enumitem}
\usepackage{mathtools}
\usepackage{nicefrac,setspace}
\usepackage{xcolor}

\usepackage{bm}
\newcommand{\boldltwo}{{\bm{\ell}_2}} 

\usepackage{hyperref}       
\usepackage{url}            
\usepackage{booktabs}       
\usepackage{amsfonts}       
\usepackage{nicefrac}       
\usepackage{microtype}      

\newtheorem{theorem}{Theorem}[section]
\newtheorem{assumption}[theorem]{Assumption}

\newtheorem*{theorem*}{Theorem}
\newtheorem*{claim*}{Claim}

\newtheorem*{lemma*}{Lemma}
\newtheorem{defn}[theorem]{Definition}
\newtheorem{notation}[theorem]{Notation}
\newtheorem{lemma}[theorem]{Lemma}
\newtheorem{corollary}[theorem]{Corollary}
\newtheorem*{corollary*}{Corollary}
\newtheorem{obs}[theorem]{Observation}
\newtheorem{prop}[theorem]{Proposition}
\newtheorem{problem}[theorem]{Problem}

\newcommand{\Z}{{\mathbb Z}}
\newcommand{\N}{{\mathbb N}}

\newcommand{\R}{\mathbb{R}}

\newcommand{\E}{\mathop \mathbb{E}}

\newcommand{\showcomments}{yes}

\newsavebox{\commentbox}
{\ifthenelse{\equal{\showcomments}{yes}}%
{\footnotemark
    \begin{lrbox}{\commentbox}
    \begin{minipage}[t]{1.25in}\raggedright\sffamily\tiny
    \footnotemark[\arabic{footnote}]}
{\begin{lrbox}{\commentbox}}}%
{\ifthenelse{\equal{\showcomments}{yes}}%
{\end{minipage}\end{lrbox}\marginpar{\usebox{\commentbox}}}
{\end{lrbox}}}

\subjclass[2020]{
60B20, 
 05C80, 
20E07, 
20E05, 
}
\keywords{Growth of groups, free groups, random graphs, non-backtracking matrix, configuration model}
\date{\today}

\thanks{ M.L. received funding by the Hellenic Foundation for Research and Innovation (H.F.R.I.) under the
“First Call for Research Projects to support Faculty members and Researchers and the procurement
of high-cost research equipment grant”, Project Number: 1034, and from the European Union’s
Horizon 2020 research and innovation program under the Marie Sklodowska-Curie grant agreement
No 101034255. 
D.T.W research supported by NSERC.}

\title[Non backtracking matrix]{On the spectral radius of the non backtracking matrix of the configuration model
}
\author{Michail Louvaris}
\address{Dept of Math, Yale, New Haven, USA}
\email{michail.louvaris@yale.edu}

\author{Daniel T. Wise}
\address{Department of Mathematics, Weizmann Institute of Science, Rehovot, Israel
}
\email{daniel.wise@weizmann.ac.il}

\author{Gal Yehuda}
\address{Dept of Math, Yale, New Haven, USA}
\email{gal.yehuda@yale.edu}

\begin{document}
\vspace*{-1cm} 


\begin{abstract}
We prove a concentration result for the leading eigenvalue of the non--backtracking matrix of the configuration model under the assumption of uniformly bounded degrees. Let $P$ denote the limiting degree distribution. Assuming  polynomial approximation, we show that as the number of vertices tends to infinity, the leading eigenvalue of the non--backtracking matrix concentrates around
\[
\frac{\mathbb{E}[P(P-1)]}{\mathbb{E}[P]}.
\]
This quantity corresponds to the mean offspring number of the excess--degree branching process associated with the local limit of the configuration model. As a byproduct of our work we explain how this result can be applied to prove the density of the growth rates of the subgroups of the free group.
\end{abstract}

\maketitle

\vspace{-.4cm}
\tableofcontents



\section{Introduction}
Spectral methods play a central role in the study of large random graphs and complex networks. Traditionally, much of the literature has focused on the eigenvalues of the adjacency matrix or graph Laplacian. However, in sparse graphs these operators often suffer from localization effects and degree heterogeneity, which can obscure the structural information encoded in the spectrum. In recent years, the \emph{non--backtracking matrix} (also known as the Hashimoto matrix) has emerged as a powerful alternative operator for analyzing sparse networks.

The non--backtracking matrix is defined on the set of directed edges of a graph and encodes transitions between edges that do not immediately reverse direction. Consequently, powers of this matrix count non--backtracking walks, and its spectral properties capture structural information about the graph while avoiding many of the degeneracies present in adjacency spectra. The operator was originally introduced by Hashimoto in connection with the Ihara zeta function of a graph, which relates non--backtracking walks to algebraic invariants of graphs \cite{hashimoto1989zeta,terras2011zeta}.

Beyond its algebraic origins, the non--backtracking operator has become an important tool in spectral graph theory and network science. In particular, its spectrum has proven useful for community detection in sparse networks, where classical spectral algorithms based on adjacency matrices fail see \cite{gulikers2016non}. A series of influential works showed that the leading eigenvectors of the non--backtracking matrix allow efficient detection of communities in stochastic block models down to the information--theoretic threshold, a phenomenon sometimes referred to as ``spectral redemption'' \cite{krzakala2013spectral,bordenave2015non}.

The largest eigenvalue of the non--backtracking matrix also has important interpretations for dynamical processes on networks. In epidemic and percolation models, the inverse of the leading eigenvalue provides an approximation to epidemic thresholds and bond percolation thresholds in networks \cite{karrer2014percolation}. Moreover, eigenvectors associated with the leading eigenvalue lead to non--backtracking centrality measures, which mitigate the localization effects that affect classical eigenvector centrality in heterogeneous networks.

From a theoretical perspective, understanding the behavior of the largest eigenvalue of the non-backtracking matrix is of independent interest. Substantial progress has been made in characterizing the spectrum of the non--backtracking matrix for random graphs. Bordenave, Lelarge and Massouli\'e \cite{bordenave2015non},see also the references therein.  Since many sparse random graph models are locally tree--like, the idea is that their spectral properties should be closely related to Galton--Watson branching processes describing the local exploration of the graph.
We also refer to the recent \cite{stephan2022non} and \cite{bordenave2023detection}. Furthermore the non-backtracking matrix has been used to establish bounds for the spectral radius of the adjacency matrix of random graphs, see for example \cite{benaych2020spectral} and \cite{dumitriu2022extreme}. 
Spectral analysis of the non-backtracking matrix was also used to prove Friedman’s second eigenvalue Theorem \cite{bordenave2015new}.  

The configuration model provides a natural setting in which to investigate these ideas. It generates random graphs with a prescribed degree distribution and serves as a canonical model for sparse networks with heterogeneous degrees \cite{bollobas2001random,van2016random}. In this setting, heuristic arguments based on branching process approximations suggest that the spectral radius of the non--backtracking matrix should concentrate around
\[
\rho(B) \approx \frac{\mathbb{E}[P(P-1)]}{\mathbb{E}[P]},
\]
where $P$ denotes the degree distribution. This quantity corresponds to the mean offspring number of the excess--degree branching process that arises when exploring the neighborhood of a random edge.

In this paper we establish a rigorous version of this heuristic. Under the assumption of uniformly bounded degrees and a polynomial approximation, we prove that the leading eigenvalue of the non--backtracking matrix of the configuration model concentrates around
\[
\frac{\mathbb{E}[P(P-1)]}{\mathbb{E}[P]}.
\]
 These results contribute to the understanding of spectral properties of sparse random graphs and provide theoretical support for the use of the non--backtracking spectrum in network inference and dynamical processes.  Our approach is heavily influenced by tools developed in \cite{bordenave2015non} and \cite{bordenave2015new}.  
\section{Original motivation from geometric group theory}

The original motivation for this work comes from geometric group theory: we originally wanted to prove that the spectrum of growth rates of subgroups of the free group $F_r$ is dense in the interval $(1,2r-1)$. This was later proven in \cite{coulon2026growth} using a more straightforward approach, and in a more general setting including hyperbolic and surface groups.

The free group $F_r$ can be realized as the fundamental group of the bouquet $B_r$ of $r$ loops, and its Cayley graph with respect to a free basis is a $(2r)$-regular tree. This tree exhibits exponential growth: the number of vertices at distance $R$ from a basepoint grows on the order of $(2r-1)^R$. Subgroups of $F_r$ correspond to covering spaces of $B_r$, and their geometry can be studied via the growth of the corresponding covering graphs. The aim of the present paper was to recover the density of possible growth rates using probabilistic methods, by constructing suitable finite graphs and analyzing their typical behavior; the probabilistic inputs will be developed later.

A key tool is the correspondence between subgroups and graph immersions, originating in the work of Stallings. Any finite graph $G$ equipped with an immersion into $B_r$ gives rise to an injective map on fundamental groups, and hence to a subgroup of $F_r$. Moreover, any finite graph whose degrees are bounded by $2r$ can be oriented and labelled so as to admit such an immersion. The universal cover $\widetilde G$ is then a tree which embeds into the Cayley graph of $F_r$, and the growth rate of the subgroup is equal to the growth rate of $\widetilde G$.

The growth rate $\omega(\widetilde G)$ measures how rapidly the number of vertices expands with distance: if $B_R$ denotes the set of vertices of $\widetilde G$ within distance $R$ from a root, then
\[
\omega(\widetilde G) = \lim_{R \to \infty} |B_R|^{1/R}.
\]
Equivalently, this describes the exponential growth of non-backtracking paths in $G$, reflecting the branching structure of the tree $\widetilde G$. This viewpoint is encoded by the non-backtracking matrix $B_G$, whose entries record allowed transitions between directed edges without immediate reversal. A fundamental fact is that the leading eigenvalue satisfies $\lambda_1(B_G) = \omega(\widetilde G)$, so that controlling growth reduces to controlling the spectral radius of $B_G$.

Given Theorem~\ref{thm:probabilistic_claim}, one has the following corollary: let $r \geq 2$. For every $\omega_0 \in (1,2r-1)$ and every $\epsilon > 0$, there exists a subgroup $H \leq F_r$ such that $|\omega(H) - \omega_0| \leq \epsilon$. The proof proceeds by constructing finite graphs whose non-backtracking spectral radius approximates $\omega_0$ using random graphs with prescribed degree distributions (as developed in Theorem~\ref{thm:probabilistic_claim}), and then realizing these graphs as subgroups via immersions into $B_r$.

One may find more details about this method in \cite{coulon2026growth}. Moreover, in \cite{louvaris2025subgroups}, given the density of the subgroups, an explicit graph with every growth rate is constructed.
\section{Model and main result}
In this section we present our main result and give a description of how it is proven.

Let $G = (V, E)$ be a finite graph, and let 
 \[\{A\}_{u,v \in V}=\mathbf{1}\left( \{u,v\} \in E \right)\]
be its adjacency matrix. 
Denote by $\overrightarrow{E}=\{(u,v):\{u,v\}\in E\}$ the set of oriented edges of $G$.

\begin{defn}[Non-backtracking matrix]\label{defn_of_non-backtracking}
    The \emph{non-backtracking matrix} of $G$ is the $|\overrightarrow{E}|\times|\overrightarrow{E}|$ matrix  $B_G$  with entries:
    \[\{B_G\}_{\overrightarrow{e},\overrightarrow{f}}=\mathbf{1}\{\overrightarrow{e}=(u,v) \text{ and } \overrightarrow{f}=(z,w) \text{ such that } v=z \text{ and } u\neq w \}. \]
\end{defn}

By the Perron-Frobenious Theorem, the leading eigenvalue $\lambda_1(B_G)$ of $B_G$ is a non-negative real number.

We also introduce the notion of \textit{high probability events}.
\begin{notation}
 In what follows, given a probability space $(\Omega,\mathcal{A},\mathbb{P})$ and a sequence of events $A_n$, we say $A_n$ \emph{holds with high probability} if  
 \[\lim_{n \to \infty}\mathbb{P}(A_n)=1.\]
 \end{notation}
\subsection{The model}
Next we formally define the model.
For further details, see \cite[sec~1.4]{bordenave2016lecture}.

Let $2 \leq k_{\min} < k_{\max}$ be two positive integers, and let $\mathbf{d}=\{k_{\min},\ldots, k_{\max}\}$ be the set of possible degrees in the graph.
For each $n \in \N$, let $\mathbf{d}_{n}=\{n_i\}_{i \in \mathbf{d}}$ be a set with:
\begin{itemize}
    \item $\sum_{i \in \mathbf{d}}n_i=n$.
    \item $\sum_{i \in \mathbf{d}}in_i$ is even.
\end{itemize}

\begin{defn}[Random Graphs with a given degree sequence]\label{defn:random_graphs_model}
Let $n \in \N$ and $d_n = \{n_i\}_{i \in \mathbf{d}}$ be as above. 
Define $G_n(d_n)$ to be the uniform probability distribution over graphs with $n$ vertices, and with degrees $d_n$. 
\end{defn}

In many cases it is easier to work with the more constructive model of random \emph{multigraphs} with a given degree sequence.

\begin{defn}[Configuration model]\label{defn:configuration_model}
Let $n \in \N$ and $d_n = \{n_i\}_{i \in \mathbf{d}}$ be as above. 
Define $\Tilde{G}_n(d_n)$ to be the uniform probability distribution over multigraphs with $n$ vertices, and with degrees $d_n$. 
\end{defn}

The configuration model can be described algorithmically as follows. 
For a vertex $u$, let $d_u=\text{degree}(u)$, let $\Delta^{(n)}_u= \{(u,j), 1\leq j \leq d_u\} $ and let $\Delta^{(n)}=\cup_u \Delta^{(n)}_u$. 
The set $\Delta^{(n)}$ is the \emph{set of half-edges} of $V(\Tilde{G}_n)$. 
Then the set of multigraphs  with $n$ vertices and with degrees $d_n$ is exactly the set of matchings of the set $\Delta^{(n)}$, i.e.\ permutations of $\Delta^{(n)}$ that are their 
 own inverse and with no fixed points.  
Let $M(\Delta^{(n)})$ be the set of matchings of the $\Delta^{(n)}$. 
Given $\sigma \in M(\Delta^{(n)})$, two vertices are adjacent if two half-edges of those vertices are matched via $\sigma$. 

Instead of picking a multigraph uniformly at random, we can equivalently pick a matching in $M(\Delta^{(n)})$ uniformly at random. Lastly we have:
\begin{align*}
    |M(\Delta^{(n)})|
     \ = \ (|\Delta^{(n)}|-1)(|\Delta^{(n)}|-3)\cdots 1
     \ = \ (|\Delta^{(n)}|-1){!}{!}
\end{align*}
where $k!!$ denotes $(k)(k-2)(k-4)\cdots$.

Denote by $\mathbf{D} = \{d_n\}_{n \in \N}$ a sequence of degree sequences, by $\mathbf{G}_n = \{G_n(d_n)\}_{n \in \N}$ a sequence of random graphs, where for each $n$ we draw $G_n$ from the uniform distribution over graphs with degrees $d_n$, and by $\Tilde{\mathbf{G}}_n = \{\Tilde{G}_n(d_n)\}$ a sequence of random multigraphs, drawn from the configuration model.

\begin{obs}
For the sequence $d_n$ to be graphic, it is necessary that $\sum_{i\in \mathbf{d}} i n_i$ be even.
However this condition is not sufficient. 
See Theorem~\ref{thm:graphic,from multi to simple}. 
\end{obs}

Let $P$ be a probability distribution supported on $\mathbf{d}$, so that for all sufficiently large $n$ and for some $d\in (0,1]$: 
\begin{align}\label{rate of convergence assumption simple graph}
    \left| \frac{n_i}{n}-P(i) \right| \ \leq \ \frac{C}{n^d}, \text{ for all }  k_{\min}\leq i \ \leq \ k_{\max}.
\end{align}
\begin{obs}
Given a distribution $P$ as above,  we can construct a 
sequence $\mathbf{D}$ by observing that for $n\in \N$:
    \[n-(k_{\max}-k_{\min}) \ \leq \ \sum_{i \in \mathbf{d}} \lfloor P(i)n\rfloor \ \leq \ n. \]
    So for all $n$ and appropriate non-negative integers $\{x^{(n)}_{i}\}_{i \in \mathbf{d}}$ uniformly bounded by $k_{\max}$, the sequence $\mathbf{D}=\left\{\{\lfloor P(i)n\rfloor +x^{(n)}_i\}_{i\in\mathbf{d}}\right\}_{n \in \N}$ satisfies the assumptions of Definition~\ref{defn:random_graphs_model}. 
\end{obs}

\subsection{Main result: Largest eigenvalue of the Non-backtracking matrix}
For the leading eigenvalue of the nonbacktracking matrix of a random graph with a given degree sequence we have the following 

\begin{theorem}\label{thm:probabilistic_claim}
     Let $\mathbf{G}_n = \mathbf{G}_n(d_n)$ be a sequence of random graphs with a given degree sequence as in Definition~\ref{defn:random_graphs_model}, and let $P$ be a probability distribution as in \eqref{rate of convergence assumption simple graph}.
     
     Define
     \[\alpha  \ := \ \frac{\E[P(P-1)]}{\E [P]}.\]
     Then there exists  $c>0$, such that  with high probability 
     \[|\lambda_1 ( B_{\mathbf{G}_n})-\alpha| \ \leq \  \frac{1}{c \log n}.\]
\end{theorem}




\subsection{Proof of Theorem~\ref{thm:probabilistic_claim}}

The random graphs $\mathbf{G}_n$ have been extensively studied and, under mild assumptions on the degrees (which our model satisfies), their local limiting behavior is well understood; see for example \cite{bordenave2016lecture}. 
A convenient way to analyze this random graph model is to consider its multigraph analogue with the same degree sequence but allowing loops and multiple edges. One then conditions on the event that the resulting multigraph is simple; see Theorem~\ref{thm:graphic,from multi to simple}. 
Working with multigraphs is advantageous because they can be interpreted as matchings of a set determined by the degree sequence (see Definition~\ref{defn:configuration_model}). 
The model obtained by choosing such multigraphs uniformly at random is the \emph{configuration model}, and we denote its distribution by $\mathbf{\Tilde{G}_n}$.

To determine the asymptotic behavior of $\lambda_1(B_{\mathbf{G}_n})$, we follow the strategy developed in \cite{bordenave2015non}, which analyzes the leading eigenvalue of the non--backtracking matrix in the Erd\H{o}s--R\'enyi model and its generalization, the stochastic block model. 
In order to characterize the largest eigenvalue of the non--backtracking matrix, it suffices to understand high-order moments of the operator; see \cite[Prop.~7]{bordenave2015non}, which yields the bound in Equation~\eqref{|lambda-a|}.

Specifically, if $\chi$ is the vector with all entries equal to $1$ and $\ell = O(\log n)$, it suffices to analyze the vectors
\[
B^\ell_{\mathbf{G}_n}(B^*_{\mathbf{G}_n})^\ell \chi
\quad\text{and}\quad
B^\ell_{\mathbf{G}_n}\chi .
\]
Most entries of these vectors correspond to non--backtracking paths on a tree. 
The mean number of children of a vertex in this tree can be computed explicitly from the asymptotic behavior of the sequence $\{n_i\}$ and yields the parameter $\alpha$ appearing in Equation~\eqref{|lambda-a|}.

To apply the techniques of \cite{bordenave2015non}, we also require a bound on the remaining eigenvalues of $B_{\mathbf{G}_n}$. 
For this purpose we use techniques from \cite{bordenave2015new}, in particular from the proof of \cite[Thm.~3]{bordenave2015new}.

The proof relies on two main ingredients. 
The first is Proposition~\ref{prop:asymptotics_for_ probabilistic}, which is obtained using path-counting combinatorics. 
The second is Proposition~\ref{prop:second_eigenvalue}, which relies on a structural theorem for the non--backtracking matrix.

In what follows we write $B_n$ instead of $B_{\mathbf{G}_n}$. 
Fix $\eta>0$ with $k_{\max}^{\eta}\le a$. 
Let $\chi$ be the $|\overrightarrow{E}|$-vector with all entries equal to $1$, and let $\Tilde{\psi}$ be the $|\overrightarrow{E}|$-dimensional vector defined by
\[
\Tilde{\psi}_e = \deg(u)-1
\]
for any directed edge $e=(u,v)\in \Delta^{(n)}$.

\begin{prop}\label{prop:asymptotics_for_ probabilistic}
Let $\ell=\delta_0\log_a n$ with $\delta_0<\frac{\eta d}{16}$. Then the following hold with high probability:
\begin{align*}
c_0 a^\ell \le 
\frac{\|B_n^\ell(B_n^*)^\ell\Tilde{\psi}\|}{\|(B_n^*)^\ell\Tilde{\psi}\|}
\le c_1 a^\ell
\end{align*}
for some constants $c_1>c_0>0$, and
\begin{align*}
\frac{\left|\langle B_n^\ell(B_n^*)^\ell \Tilde{\psi}, (B_n^*)^\ell\Tilde{\psi}\rangle\right|}
{\|B_n^\ell(B_n^*)^\ell \Tilde{\psi}\|\|(B_n^*)^\ell\Tilde{\psi}\|}
\ge c_0 .
\end{align*}
\end{prop}

Proposition~\ref{prop:asymptotics_for_ probabilistic} is proved in Proposition~\ref{concusion_for_simple}.

\begin{prop}\label{prop:second_eigenvalue}
Let $\ell=\delta_0\log_a n$ with $\delta_0<\frac{\eta d}{16}$. 
There exists $c>0$ such that with high probability
\[
\sup_{\|x\|=1:\,\langle (B^*_{\mathbf{G}_n})^*\Tilde{\psi},x\rangle=0}
\|(B^*_{\mathbf{G}_n})^\ell x\|
\le a^{\ell/2}\log^c n .
\]
\end{prop}

Proposition~\ref{prop:second_eigenvalue} is proved in Section~\ref{sectionmatriexpansion}.

\begin{proof}[Proof of Theorem~\ref{thm:probabilistic_claim} given Propositions~\ref{prop:asymptotics_for_ probabilistic} and~\ref{prop:second_eigenvalue}]
The argument follows closely the proof of \cite[Thm.~3]{bordenave2015non}, given \cite[Prop.~10 and~11]{bordenave2015non}. 
We sketch the proof for completeness.

Our goal is to apply \cite[Prop.~7]{bordenave2015non}. 
Define
\[
\phi := \frac{(B_n^*)^\ell \Tilde{\psi}}{\|(B_n^*)^\ell \Tilde{\psi}\|}, 
\qquad
\theta := \frac{\|B_n^\ell(B_n^*)^\ell\Tilde{\psi}\|}{\|(B_n^*)^\ell\Tilde{\psi}\|},
\qquad
\zeta := \frac{B_n^\ell(B_n^*)^\ell \Tilde{\psi}}{\|B_n^\ell(B_n^*)^\ell \Tilde{\psi}\|}.
\]

Let $R := B_n^\ell - \theta \zeta \phi^*$. 
For $\psi\in \mathbb{R}^{|\Delta^{(n)}|}$ with $\|\psi\|=1$, write $\psi=\psi_1+\psi_2$ where $\psi_1=s\phi$ for some $s\in\mathbb{R}$ and $\psi_2$ lies in the orthogonal complement of $\phi$. 
Then
\begin{align}\label{ineq_R_for_the_main_theorem}
\|R\psi\|
= \|B_n^\ell \psi_2 + s(B_n^\ell\phi-\theta\zeta)\|
\le
\sup_{\|x\|=1:\,\langle (B^*_{\mathbf{G}_n})^*\Tilde{\psi},x\rangle=0}
\|(B^*_{\mathbf{G}_n})^\ell x\|.
\end{align}

Combining \eqref{ineq_R_for_the_main_theorem} with Propositions~\ref{prop:asymptotics_for_ probabilistic} and~\ref{prop:second_eigenvalue}, we may apply \cite[Prop.~7]{bordenave2015non} to conclude the proof.
\end{proof}
\section{Local analysis}\label{sectionlocal}

The goal of this section is to give a local description of the elements of the non-backtracking matrix. Firstly we analyze the limiting candidate of the non-backtracking matrix, some functionals of the Unimodular Galton-Watson Process.   
\subsection{Unimodular Galton-Watson Process}
This subsection is an analogue of \cite[Sec~8]{bordenave2015non}, adjusted to our model.

Let $\Z^+ = \{z\in \Z : z\geq0\}$ and $\N = \{n\in \Z : n\geq 1\}$. 

\begin{defn}[Unimodular-Galton Watson Process]\label{unimodular galton defn}
Let $P$ be a probability measure supported on $\Z^+$ with finite mean. 
Let $\N^{f}=\bigcup_{i=0}^{\infty}\N^{i}$ be the set of all finite subsets of $\N$ with the convention that $\N^{0}=\emptyset$.  For each $\mathbf{i} \in \N^{f}$ assign a random variable $Y_{\mathbf{i}
}$ to $\mathbf{i}$ with distribution $P$ if $i=0$, and otherwise with distribution 
\begin{align}\label{barP}
    \Bar{P}(k) \ := \  \frac{(k+1) P(k+1)}{\sum_{j \in \Z^{+}}jP(j)}
\end{align}
A \emph{Unimodular Galton-Watson Process with distribution $P$} is the rooted tree $\mathcal{T}$  
with root  $\emptyset$ which we henceforth denote by $o$, and with vertices 
\begin{align}\label{vertices galtonwatson}
V \ = \ \{o\}\cup\left\{\mathbf{i}=(i_1,\ldots,i_k)\in \N^{f}: \text{ for } 1\leq l\leq k \text{ , } i_{\ell}\leq Y_{i_1 \cdots i_{\ell-1}}\right\}
\end{align}
and with edges $ \{\mathbf{i},\mathbf{i}'\}$ 
\begin{align}\label{edges galtonwatson}
    \text{ where either } \mathbf{i}=(\mathbf{i'},n) 
    \ \text{ or } \ 
    \mathbf{i'}=(\mathbf{i},n), \ \text{ for some } n \in \N.
\end{align}
\end{defn}
Basically,  \eqref{vertices galtonwatson} and \eqref{edges galtonwatson} state that for $\mathbf{i} \in \N^f$ the number of children of $\mathbf{i}$ has distribution $Y_{\mathbf{i}}$. 
\begin{assumption}[The model]\label{The model}
Based on the assumptions of Theorem\ref{thm:probabilistic_claim}, it suffices to focus on the case where $P$ is nontrivial and supported on $\{k_{\min},\ldots, k_{\max}\}$, for some $k_{\min},k_{\max}\in \N$ with $1<k_{\min}< k_{\max}$.  Moreover it is clear that 
\begin{align}\label{assumption model}
   1 \ \leq \  k_{\min}-1 \ < \ \E \bar{P} \ = \ \frac{\E P(P-1)}{\E P}
\end{align}
where ${\E P(P-1)}$ means the expectation of a random variable with distribution $P(P-1)$.  
\end{assumption}
\begin{obs}
    Assumption \ref{The model} is satisfied
    for any distribution $P$ supported on $\{2,\ldots,2r\}$ for some $r \in \N$.
\end{obs}
Given a unimodular Galton-Watson Process $\mathcal{T}$ with distribution $P$, we let $Z_t$ denote the number of vertices at generation $t$. That is, $Z_t$ is the number of vertices at distance $t$ from $o$. Let $a= \E \bar{P}$ and $b=\E P$.
\begin{obs}\label{obs bound Zt}
    For all $t\in \N$
    \begin{align*}
    (k_{\min})^{t} \ \leq  \   Z_t \ \leq \ (k_{\max})^t
    \end{align*}
    almost surely.
\end{obs}
\begin{proof}
    By \eqref{barP}, each generation has at least $k_{min}$
    and at most $k_{\max}$ children.
\end{proof}
\begin{lemma}\label{Z_t matringale}
    The process $\frac{Z_t}{a^{t}}$ is a martingale with respect to the filtration $(Z_1,\ldots, Z_{t-1})$. 
\end{lemma}
\begin{proof}
    By Observation~\ref{obs bound Zt}, it suffices to show that for each $t\geq 2$:
    \begin{align}\label{induction Z_t}
        \frac{\E {Z_t}\big|(Z_1,\ldots, Z_{t-1})}{a}=  Z_{t-1} 
    \end{align}
    One can compute that for  $t\geq 2$:
    \begin{align}\label{Z_t=sum}
       Z_{t}=\sum_{i=1}^{Z_{t-1}} \xi_{i}
    \end{align}
    where $\xi$ are i.i.d. random variables with law $\Bar{P}$.  The claim follows by taking the condition expectation of $Z_t$ with respect to $(Z_1,\ldots, Z_{t-1})$ in \eqref{Z_t=sum}.
\end{proof}
\begin{obs}\label{obs given history}
    Let $u \in V(\mathcal{T})$ be at distance $t$ from the root. Let $Z^{(u)}_{m}$ be the number of offspring of $u$ at distance $m$. Then  the process $Z^{(u)}_{m}\mid\left(Z_1, \ldots, Z_{t} \right)$ behaves as the process $Z_m$ except at the first generation.
\end{obs}
\begin{obs}
    By \eqref{induction Z_t} we have $\E Z_t=a^{t-1}b$.
\end{obs}
In general we are interested in the asymptotic behavior of the functional of $\frac{Z_t}{a^t}$ as $t\to \infty$. 
The following gives a first estimate on $\frac{Z_t}{a^t}$. 
\begin{lemma}\label{lemma limit of Zt} 
The sequence of random variables
\begin{align*}
     \frac{Z_{t}}{a^{t}}
\end{align*}
converges almost surely to a positive random variable.  Moreover by the dominated convergence theorem the convergence also occurs on $L_{2}$. 
\end{lemma}
\begin{proof}
By Lemma~\ref{Z_t matringale} and Doob's martingale Theorem, it suffices to prove the $L_2$ norm of $a^{-t}Z_t$ is uniformly bounded for all $t \in \N$.
To that end, note for any $t \in \N$, if $\mathcal{F}_t$ denotes the filtration produced by the random variables $\{Z_{i}\}_{i\leq t}$ we have:
\begin{align*}
    \E \left( \frac{Z_{t+1}}{a^{t+1}} - \frac{Z_t}{a^t} \right)^2 \ = \ a^{-2(t+1)} \E \left(\E (Z_{t+1}-a Z_t)^2 \mid \mathcal{F}_t\right)
    \\= a^{-2(t+1)} \E \E \left(\operatorname{Var}(\Bar{P}) Z_{t}  \mid \mathcal{F}_t\right) \ = \ \frac{\operatorname{Var}(\Bar{P})b}{a}   \left(\frac{1}{a}\right)^t   
\end{align*} 
In particular, for some absolute constant $C>0$

\begin{equation*}
    \left|\frac{Z_t}{a^t}\right|_\boldltwo
    \ \leq \ 
    \left|Z_{1}\right|_\boldltwo + \sum_{s=0}^{t-1}\left|\frac{Z_{s+1}}{a^{s+1}}- \frac{Z_s}{a^s}\right|_\boldltwo
    \ \leq \
 k^2_{\max} +  C \sum_{s=1}^{\infty} \left( \frac{1}{a} \right)^{s/2} \ < \ \infty.  \qedhere
\end{equation*}
\end{proof}
Although Observation~\ref{obs bound Zt} already gave a deterministic bound for $Z_t$,  our approach requires the following better bound:
\begin{lemma}\label{bound on Z_k<a^ks}
    There exist constants $c_0,c_1>0$ such that for any $s\geq k_{\max}+1$  
    \begin{align*}
        \mathbb{P}\left( Z_{k}\leq a^{k} s  \text{ for all } k\in \N \right)
        \ \geq \ 1-c_0 \exp\left( - c_1 s \right).
    \end{align*}
\end{lemma}
\begin{proof}
Define
\begin{equation*}
    \epsilon_k=\frac{k^{1/2}}{a^{k/2}} \text{ \ \ and \ \ } f_k=\prod_{i=1}^{k} (1+\epsilon_i)
\end{equation*}
It is easily shown that $\lim_{k\to \infty}f_k<\infty$. Thus  we can fix parameters $c'_0,c'_1>0$ with
\begin{align*}
    c'_0<f_{k}<c'_1 \text{ \ and \ } \epsilon_k <c'_1 \text{ \ for each  } k\in \N.
\end{align*}
By Hoeffding’s inequality   \cite[Thm~2.2.5]{vershynin2018high}, for any $k \in \N$, if $\{\xi_{i}\}_{i \in [k]}$ are i.i.d. random variables with law $\Bar{P}$, then for any $s>1$:
\begin{align}\label{hoeffding growth Z_t}
   \mathbb{P}\left( \sum_{i=1}^{k} \xi_{i}\geq k s a \right) \ \leq \ \exp \left(- \frac{2 k a^2 (s-1)^2}{(k_{\max}-k_{\min})^2} \right) 
\end{align}
Define the events
\begin{align*}
    \Omega(k)= \{ Z _k \geq a^k s f_k \}
\end{align*}
For $k=1$ and $s> k_{\max}$ it is clear that $\mathbb{P}(\Omega(1))=0$. Then for any $t \geq 1$, if $\mathcal{F}_t$ is the $\sigma$-algebra generated by $(Z_1,\ldots, Z_t)$, then for any $k\geq 2$ by \eqref{hoeffding growth Z_t} we have:
\begin{multline}\label{bound for omega cap}
   \mathbb{P} \left[\prod_{i \leq k-1} \mathbf{1}\left( \Omega^{c}(i) \right) Z_{k}\geq f_{k+1} s a^k \mid \mathcal{F}_{k-1} \right] \leq \exp\left(- \frac{2sa^{k-1} f_k a^2\epsilon^2_k}{(k_{\max}-k_{\min})^2} \right) 
   \\ \leq 
   \exp \left(-sk \frac{c'_0 a }{(k_{\max}-k_{\min})^2} \right).
\end{multline}
Thus for any $k>1$,
\begin{equation*}
    \left|\mathbb{P}\left( \bigcap_{i \leq k}\Omega^c(i) \right) - \mathbb{P}\left( \bigcap_{i \leq k}\Omega^c(i) \right)
    \right|
    \ \leq \
    \exp \left(- c_0 s k \right)
\end{equation*}
for some absolute constant $c_0>0$. Thus,
\begin{equation}\label{inductin omega cap}
    \left|\mathbb{P}\left( \Omega^c(1) \right) - \mathbb{P}\left( \bigcap_{i \leq k+1}\Omega^c(i) \right) \right| 
    \ \leq \ \sum_{i=2}^{k} \exp(-c_0 s i)
    \ \leq \
    c_1 \exp (-c_0 s).
\end{equation}
Now the claim follows by \eqref{inductin omega cap}. 
\end{proof}

The following shows that 
 the process $Z_{t}$
has an almost exponential growth:
\begin{prop}\label{growth of Zt}
 For any $\gamma>0$, there exists some constant $C=C(\gamma)$ such that with probability $1-\frac{1}{n^\gamma}$, 
 for any $t,s \in \N$ with $s<t$ we have
    \begin{align*}
        |a^{s-t} Z_t - Z_s |
        \ \leq \
        C (s+1) a^{\frac{s}{2}} (\log n)^{C} 
    \end{align*}
    for some absolute constant $c>0$.
\end{prop}
\begin{proof}
Let $\mathcal{F}_t$ be the $\sigma$-algebra generated by $(Z_1,\ldots, Z_t)$. 
 By Hoeffding's inequality \cite[Thm~2.2.5]{vershynin2018high}, we have:
    \begin{align*}
        \mathbb{P}
        \bigg(|Z_{t+1}-aZ_{t}| \ \geq \  Z^{1/2}_{t} \lambda \bigg| \mathcal{F}_t \bigg)
        \ \leq \ 
        2 \exp\left(- \frac{\lambda^2}{(k_{\max}-k_{\min})^2}\right)
    \end{align*}
    Letting $\lambda= u (t+1) \log^{c'} n$, there exist constants $c_0,c_1$ such that with probability $1-c_0 \frac{1}{n^{c_{1}u}}$, we have $|Z_{t+1}-a Z_{t}| \leq u (t+1) \log^c n Z^{1/2}_t$ for any $t\geq 1$. 
    
     By Lemma~\ref{bound on Z_k<a^ks}, for an appropriate choice of $C=C(\gamma)$ the event $Z_{t}\leq Ca^{t}\log n$ for any $t\in \N$, holds with probability $1-c_1 \frac{1}{n^{c_0\gamma}}$, for some absolute constants $c_0,c_1>0$.
    
    Thus on the intersection of those events,
    \begin{align}\label{ineq Z_t-Z_s}
        a^s \left|\frac{Z_t}{a^t}-\frac{Z_s}{a^s}\right|
        \ \leq \ 
        a^s \log^c n  \sum_{u=s}^{t-1} (u+1)  (a)^{u/2} a^{-(u+1)} 
        \ \leq \
        C a^{\frac{s}{2}} (s+1)\log^Cn, 
    \end{align}
    where the last inequality of \eqref{ineq Z_t-Z_s} used that $\sum_{m\geq s}md^{m}\leq c d^{s}s$ for any $d\in (0,1)$. 
\end{proof}

We can now present asymptotic results for the functionals with which we will compare the largest singular value of the non-backtracking matrix.

For a unimodular Galton-Watson tree $\mathcal{T}$ with distribution $P$, define $\mathcal{Y}^{(u)}_\ell $ to be the set of walks $(v_1,v_2,\ldots,v_\ell )$ starting from the vertex $u$ and so that $u \neq v_i$ for any ${\ell} \geq i>1$.
Recall that $o$ denotes the root
of $\mathcal{T}$. 
For ${\ell}\in \N$ define
\begin{align}\label{denf Q_l}
    Q_{\ell}= \left|\{ \text{walks in } \mathcal{Y}^{(o)}_{2\ell+2}  \text{ that backtrack once so that } v_{\ell-1}=v_{\ell+1}\}\right|.
\end{align}
 
 One can rewrite $Q_{\ell}$ based on which level $0\leq t \leq \ell-1$ (after the backtracking) the path will lead to a new vertex. So if $Z^{(v)}_{t}$ is the number of offspring of a vertex $v $ at the $t$-th generation, one has  
\begin{align}\label{Q_l}
    Q_{\ell}= \sum_{t=0}^{\ell-1} \sum_{v \in Y^{(o)}_{t}} L^{v}_{\ell}
\end{align}
where $Y^{(u)}_{t}$ is the offspring of a vertex $u$ at distance $t$ and
\begin{align*}
    L^{v}_{\ell}= \sum_{w\neq u \in Y^{(v)}_{1}} Z^{(w)}_{\ell-t-1} Z^{(u)}_{t+1} 
\end{align*}
\begin{prop}\label{thm 25}
    The random variable $\frac{Q_{\ell}}{a^{2\ell}}$ converges as ${\ell}\to \infty $ a.s.\ and on $L_2$ to $C Z_{\infty}$ for some absolute constant $C>0$. Here $Z_{\infty}$ is the limit of the random variables $\frac{Z_t}{a^t}$ from Lemma~\ref{lemma limit of Zt}.
\end{prop}
\begin{proof}
Let $\mathcal{F}_{t}$ be the filtration generated by $(Z_1,\ldots, Z_t)$. 
    It suffices to prove:
    \begin{enumerate}
        \item Let $\Bar{Q}_\ell=\sum_{t=0}^{\ell-1}\E_{\mathcal{\mathcal{F}}_{t+1}}  \sum_{v \in Y^{(o)}_{t}} L^{v}_{\ell}$. The Proposition's claim holds for $\Bar{Q}_{\ell}$
        \item $\E (Q_{\ell}-\Bar{Q}_{\ell})^2=o(a^{2\ell})$
    \end{enumerate}
    For the proof of claim~(1) note that given $\mathcal{F}_{t+1}$ the random variables $Z^{w}_{\ell-t-1}$ and $Z^{u}_{t+1}$ are independent for any $w,u \in Y^{(o)}_{t+1}:w\neq u$. Note that $\E_{\mathcal{F}_{t+1}} Z^{x}_{y}=a^y$ for $x \in \{u,w\}$, and $y \in \{t+1, \ell-t-1\}$. So 
    \begin{align}\label{Qbar}
      \bar{Q}_\ell=  a^{\ell} \sum_{t=0}^{\ell-1} \E_{\mathcal{F}_{t}} \sum_{v \in Y^{(o)}_{t}}  \E_{\mathcal{F}_{t+1}} (Z^{(v)}_{1}-1) Z^{(v)}_{1}= \sum_{t=0}^{\ell-1}a^{\ell} (\operatorname{Var}(\Bar{P})+a^2-a) Z^{o}_{t}
    \end{align}
    The constant $ (\operatorname{Var}(\Bar{P})+a^2-a)$ appears since $\E_{F_{t+1}} Z_{1}^{v}(Z_{1}^{v}-1)= \E \Bar{P}(\Bar{P}-1))$ where $\Bar{P}$ is a random variable with law described in Definition~\ref{unimodular galton defn}.
    
    So by \eqref{Qbar}, it is implied that $\bar{Q}_{\ell}a^{-2l}$ converges to $C Z_{\infty} \frac{a}{(a-1)}$ both a.s.\ and on $L_2$. 

   For a random variable $Y$ let $\| Y \|_{\ell _2}=(\E Y^2)^{1/2}$. For claim~(2) note that:
    \begin{align}\label{Q-barQ}
        \|Q_{\ell}-\bar{Q}_{\ell}\|_{\ell _2} 
        \ \leq \ \sum_{t=0}^{\ell-1} \| \sum_{v \in Y^{(o)}_{t}} L^{v}_{\ell}- \E_{\mathcal{F}_{t+1}} L^{v}_{\ell} \|_{\ell _2} 
        \ = \ \sum_{t=0}^{\ell-1}\| (\operatorname{Var}_{\mathcal{F}_{t+1}} \sum_{v \in Y^{(o)}_{t}} L^{v}_{\ell} )^{1/2}\|_{\ell _2} 
    \end{align}
    Moreover we have the following bound
    \begin{align}\label{ineq VAR L}
        \operatorname{Var}_{\mathcal{F}_{t+1}} \sum_{u \in Y^{(o)}_{t}} L^{u}_{\ell} \leq 
        \sum_{u \in Y^{(o)}_{t}}  \operatorname{Var} L^{u}_{\ell} \leq   \sum_{u \in Y^{(o)}_{t}}  \E_{\mathcal{F}_{t+1}}(L^{(u)}_{\ell})^2 \leq    \sum_{u \in Y^{(o)}_{t}}  \E_{\mathcal{F}_{t+1}}(Z^{(u)}_{t})^2 \E_{\mathcal{F}_{t+1}} (Z^{(u)}_{\ell-t-1})^2 \leq C Z^{(o)}_{t} a^{2\ell-2} 
    \end{align}
    for some absolute constant $C>0$. The second to last inequality of \eqref{ineq VAR L} used the Cauchy-Schwartz inequality. The last inequality used Lemma~\ref{Z_t matringale} and Observation~\ref{obs given history}.  By combining \eqref{ineq VAR L} and \eqref{Q-barQ} we obtain: 
    \begin{equation*}
        \|Q_{\ell} -\bar{Q}_{\ell} \|_{\ell _2} \ 
        \leq \  
        C a^{\ell-1} \sum_{t=0}^{\ell-1} (\E Z^{(o)}_{t})^{1/2}  
        \ = \ 
        C a^{\ell-1} \sum_{t=0}^{\ell-1} b^{1/2} a^{(t-1)\vee 0} 
        \ = \
        o(a^{2\ell})
        \qedhere
    \end{equation*}
\end{proof}

\begin{corollary}
For any $p>1$ there exists $ C=C(a,p)$  such that $\E Q^{p}_{\ell} \leq C a^{2\ell p}$.
\end{corollary}
\begin{proof}
By Lemma~\ref{bound on Z_k<a^ks} and the fact that $\E X^p=\int_{0}^{\infty} t^{p-1}\mathbb{P}(X\geq t) dt$ for any random variable $X\geq 0$ a.s.\ and $p\geq 1 $, 
the following holds for some absolute constant $C=C(a,p)$:
\begin{align}\label{Z_t^p<Ca^tp}
    \E Z^{p}_{t}\leq C a^{tp}
    \end{align}
    Thus, two applications of the inequality $|\sum_{i=1}^{n}x_i|^p \leq n^{p-1} \sum_{i=1}^n |x_i|^p $ yields:
    \begin{align*}
        \E_{\mathcal{F}_t} \left| \sum_{u \in Y_{t}^o}  L^{u}_{\ell}\right|^p  
        \ \leq \ 
        C (Z^{(o)}_{t})^{p-1}\sum_{u \in Y_{t}^{(o)}} a^{(\ell-1)p}\E_{\mathcal{F}_t} (Z^{(u)}_1)^{2(p-1)} 
        \ \leq \ 
        C' (Z^{(o)}_{t})^p a^{\ell p}.
    \end{align*}
    Thus,
\begin{equation*}
    (\E Q_{\ell}^p)^{1/p} 
    \ \leq \ 
    \sum_{t=0}^{\ell-1}\left( \E \left| \sum_{u \in Y_{t}^o}  L^{u}_{\ell}\right|^p  \right)^{1/p}
    \ \leq \
    C'' a^{\ell} \sum_{t=0}^{\ell-1} a^t
    \ = \ 
    \mathcal{O}(a^{2\ell})
\qedhere
\end{equation*}
\end{proof}
\begin{obs}\label{obs different start}
    Our results focus on $Z_t$ when $t$ is large. In particular we don't need $Z_1$ to follow $P$, but instead, $Z_1$ can follow any distribution compactly supported on a subset of $\Z^+$ that is larger than $0$ 
    with probability~1. Thus our  later analysis can use a different distribution for the number of children at the first generation $Z_1$, and the results will still hold.  
\end{obs}
\subsection{Local analysis for the configuration model}\label{subsection local analysis}
In this subsection we present the first step of our analysis of the configuration model. 
We begin by presenting some known results for the configuration model.
        

As described in Definition~\ref{defn:random_graphs_model}, the law of a uniformly chosen multigraph from the multigraphs with a given degree sequence, can be equivalently described as a uniformly chosen matching for the half-edges of the multigraph.

Fix $n\in \N$ and $\mathbf{\Tilde{G}}_n$ a random multigraph as in Definition~\ref{defn:configuration_model}. 
Fix $v\in V(\mathbf{\Tilde{G}}_n)$ and $\Delta^{(n)}=\{(u,j): u \in V(\mathbf{\Tilde{G}}_n) \text{ and } 1\leq j \leq  d_u\}$ the set of half edges of $\mathbf{\Tilde{G}}_n$. For $u\in \mathbf{\Tilde{G}}_n$ let $\Delta^{(n)}_u$ be the half edges in $\Delta$ adjacent to $u$. Define the following exploration process:
\begin{itemize}
    \item At integer step $t$, we partition the half edges of $\mathbf{\Tilde{G}}_n$ into three sets, the active set $A_t$, the unexplored set $U_t$ and the connected set $C_t$.
    \item For $t=0$ let $C_0=\emptyset$, let $A_0=\Delta^{(n)}_v$ and $U_0=\Delta^{(n)}\setminus \Delta^{(n)}_v$.
    \item For each $t\geq 1$, if $A_t\neq \emptyset$ let $e_{t}=(u,j)\in A_t$ with $u$  at minimum distance from $v$ (among vertices with half-edges in $A_t$)  and $(u,i)\notin A_t$ for all $i\leq j-1$. 
    Let $\sigma(e_t)=(w,l)$ be the image of $e_t$ through the uniformly chosen matching $\sigma$ of $\mathbf{\Tilde{G}}_n$. Let $I_{t+1}=\Delta^{(n)}_{w}\setminus\{\sigma(e_{t+1})\}\cap U_t$ and:
\begin{itemize}
    \item[] $A_{t+1}=A_t\setminus\{e_t,\sigma(e_t)\} \cup I_{t+1}$
    \item[] $C_{t+1}=C_t\cup \{\sigma(e_t),e_t\}$
    \item[] $U_{t+1}= U_t\setminus \left(I_{t+1} \cup\{\sigma(e_{t})\}\right)$
\end{itemize}
The process stops when $A_t=\emptyset$. 
\end{itemize}
    Let $X_{t} =|I_t|$, let $\epsilon_t=\mathbf{1}\left( \sigma(e_t) \in A_t  \right)$, and let $\tau_1=\inf\{t\geq0: |A_t|=\emptyset\}$.

The following hold for $t\leq \tau_1$:
\begin{align*}
    |A_t|=d_v+\sum_{k=1}^{t}(X_k-1-\epsilon_k) \text{ \ and \ } |C_t|=2t \text{ \ and \ } |U_t|=\sum_{i\in \mathbf{d}}in_i- d_v-\sum_{k=1}^{t}(X_k+1-\epsilon_k) 
\end{align*}
It is clear that $|A_t|\leq k_{\max}t$.

The following presents some well-known facts about the process described above. This will be the basis for our analysis.
\begin{theorem}\label{theorem_properties_conf_model}
\
\begin{itemize}
    \item For every $1\leq t\leq \tau_1$ every half edge in $C_t$ has been matched. Moreover the graph spanned by the edges in $C_t$ is not a tree only if $\sum_{k\leq t}\epsilon_k\neq 0$. Moreover there exist absolute constants $C,c>0$ such that: 
    \begin{align}\label{prob of a tree}
        \mathbb{P}\left( \sum_{k=1}^{ t\wedge \tau_1}\epsilon_k\neq 0 \right) \ \leq \ C \sum_{s=1}^t \frac{s}{cn-2s-1}
    \end{align}
    \item For every $t\geq 0$ if $\mathcal{F}_t$ is the filtration generated by 
    \[\left[(C_0,U_0,A_0),(C_1,U_1,A_1),\ldots, (C_t,U_t,A_t) \right]\]
    then for any $k \in \{k_{\min}-1,\ldots, k_{\max}-1\}$ we have:
    \begin{align}\label{X-t and Z_t}
        \left|\mathbb{P}\left( X_t=k\mid \mathcal{F}_t \right)- (k+1) \frac{n_{k+1}}{\sum n_i i -2t-1}  \right|  \leq \frac{Ct}{c n -2t-1}
    \end{align}
    Moreover  there exists an absolute constant $C>0$ such that:
    \begin{align}\label{epsilon_l_tangle_free}
        \mathbb{P}(\epsilon_{t}=1)\mid \mathcal{F}_t) \leq C \frac{k_{\max}^{t}}{n}
    \end{align}
    \item For a functional $F:\{\text{matchings of } \Delta^{(n)}\} \to \R^+$and a constant $c>0$ satisfying
    $|F(m)-F(m')| \ \leq \ c$ 
    whenever $m\neq m'$ are matchings whose images differ in at most $4$ half edges, we have:
    \begin{align}\label{concentration Theorem}
    \mathbb{P}\left[\left|F(\mathbf{\Tilde{G}}_n)-\E \left(F(\mathbf{\Tilde{G}}_n)\right)\right|\geq t \right] \ \ \leq \ \ 
    \exp\left(-\frac{t^2}{c^2C'n} \right)      
    \end{align}
    for some absolute constant $C'>0$.
    
\end{itemize}
\end{theorem}
\begin{proof}
    The general formulation of these results can be found in  \cite[sections 3.5.2, 3.6.3 and 3.6.4]{bordenave2016lecture}. We have adjusted them to our model.
\end{proof}

The \emph{$\ell$-neighbourhood} of a vertex $v$ consists of vertices at distance at most $\ell$ from $v$.
We are interested in the $\ell$-neighbourhood of $v \in \mathbf{\Tilde{G}}_n$ where $\ell=\mathcal{O}(c\log_a n)$, for small enough $c>0$. 
Fix $\eta>0$ such that $(k_{\max})^{\eta}<a$. Let $\ell=\delta_0\log_a n$ for $\delta_0<\frac{\eta d}{16}$. 
Fix $v \in V(\mathbf{\Tilde{G}}_n)$, and let $$\tau_2(v)=\inf_{t\geq 0}\big\{ \{l\text{-neighbourhood of }v\} \subseteq C_t  \big\}$$
Since every vertex has degree at most $k_{\max}$, we get the deterministic $\tau_2(v)\leq k^{\ell}_{\max} \leq a^{\frac{1}{8}\log_a n}=n^{1/8}$. 

\begin{corollary}\label{cor tree}
    For any $v \in V(\mathbf{\Tilde{G}}_n)$ the graph $(\mathbf{\Tilde{G}}_n(v))_{\ell}$, the $\ell$-neighbourhood of $v$, is a tree with probability 
    \begin{equation*}
        1-\frac{c}{n^{3/4}}
    \end{equation*}
    for some absolute constant $c>0$.
    \end{corollary}
    \begin{proof}
        By \eqref{prob of a tree} we have
        \begin{equation*}
            \mathbb{P}((\mathbf{\Tilde{G}}_n(v))_{\ell} \text{ is not a tree}) 
             \ \leq \ 
             C \sum_{s=1}^{n^{\frac{1}{8}}} \frac{s}{cn-2s-1}
             \ \leq \ 
             C' n^{\frac{1}{4}}\frac{1}{c' n} 
            \ \leq\ 
            C'' \frac{1}{n^{3/4}}.   
\qedhere        \end{equation*}
    \end{proof}
    For random variables $X,Y$ supported on $\Z^+$, recall their total variation distance
\[\operatorname{d_{TV}}(X,Y) \ = \ \sum_{k\in \Z^+}\frac{1}{2} \left| \mathbb{P}(X=k)-\mathbb{P}(Y=k) \right| \ = \ \min \mathbb{P}(\Tilde{X}\neq \Tilde{Y}), \]
where the minimum is taken over all couplings $(\Tilde{X},\Tilde{Y})$ with $X\sim \Tilde{X}$ and $Y\sim \Tilde{Y}$. 

Similarly for two processes $X_t, Y_t \in (\Z^+)^k$ for some $k>0$ we have that 
\[ \operatorname{d_{TV}}(X_t,Y_t)= \frac{1}{2} \sum_{x \in (\Z^+)^k} \frac{1}{2}\left| \mathbb{P}(X_t=x)-\mathbb{P}(Y_t=x)\right|= \min \mathbb{P}(\Tilde{X}_t\neq \Tilde{Y}_t), \]
where again the minimum is taken over all couplings $(\Tilde{X}_t,\Tilde{Y}_t)$ with $X_t\sim \Tilde{X}_t$ and $Y_t\sim \Tilde{Y}_t$.

We have the following bound for the total variation of two processes.
\begin{lemma}\label{lem_tv_stoch_proc}
Let $(S_t)_{t\ge 0}$ and $(Z_t)_{t\ge 0}$ be two stochastic processes taking values in the same state space. 
Assume that for every $t\ge 0$ we have
\[
\Pr\!\left(S_{t+1} \neq Z_{t+1} \mid S_0=Z_0,\ldots,S_t=Z_t\right) \le \alpha_t
\]
for some sequence $(\alpha_t)_{t\ge0}$ for which $a_0=0$.

Then for every $t\ge0$,
\[
\Pr\!\left((S_0,\ldots,S_t) \neq (Z_0,\ldots,Z_t)\right)
\le
\sum_{r=0}^{t-1} \alpha_r.
\]
In particular, if the two processes are coupled on the same probability space, then
\[
d_{TV}\!\left(\text{Law}(S_0,\ldots,S_t),\,\text{Law}(Z_0,\ldots,Z_t)\right)
\le
\sum_{r=0}^{t-1} \alpha_r.
\]
\end{lemma}

\begin{proof}
Let
\[
E_t := \{S_0=Z_0,\ldots,S_t=Z_t\}
\]
be the event that the two processes agree up to generation $t$.

Observe that
\[
E_{t+1} = E_t \cap \{S_{t+1}=Z_{t+1}\}.
\]
Therefore
\[
E_{t+1}^c
=
E_t^c \cup \big(E_t \cap \{S_{t+1}\neq Z_{t+1}\}\big).
\]
Taking probabilities gives
\[
\Pr(E_{t+1}^c)
\le
\Pr(E_t^c) + \Pr(E_t \cap \{S_{t+1}\neq Z_{t+1}\}).
\]
Now
\[
\Pr(E_t \cap \{S_{t+1}\neq Z_{t+1}\})
=
\Pr(S_{t+1}\neq Z_{t+1}\mid E_t)\Pr(E_t)
\le \varepsilon_t.
\]
Hence
\[
\Pr(E_{t+1}^c) \le \Pr(E_t^c) + \alpha_t.
\]

Since $\Pr(E_0^c)=0$, iterating the inequality yields
\[
\Pr(E_t^c)
\le
\sum_{r=0}^{t-1} \varepsilon_r.
\]

Finally, observe that
\[
E_t^c = \{(S_0,\ldots,S_t)\neq (Z_0,\ldots,Z_t)\},
\]
which proves the first claim. The bound on the total variation distance follows from the fact that under any coupling,
\[
d_{TV}(\mu,\nu) \le \Pr(X\neq Y).
\]
\end{proof}
In what follows set $S_t(v)$ be the number of 
length~$t$ non-backtracking walks starting at $v$ in $\mathbf{\Tilde{G}}_n$.
\begin{corollary}\label{cor total variation for vertex}
    Recall that $Z_t$ denotes the number of offspring of a unimodular Galton Watson process with distribution $P$ satisfying Assumption~\ref{The model}. 
    Let $\mathbf{\Tilde{G}}_n$ be the Configuration Model 
    with distribution $P$, described in Definition~\ref{defn:configuration_model}.
    
    For any $v\in \mathbf{\Tilde{G}}_n$, if $\tau_2(v)$ is the stopping time for which all the vertices at distance at most $\ell$ from $v$ have been revealed during the exploration process and 
    $${deg}(v)= k$$
    for some $k\in \mathbf{d}$, we have the following
        for some absolute constant $c>0$:
    \begin{align*}
        \operatorname{d_{TV}}\left(\{Z_t\}_{t\leq \tau_2(v)}\mid Z_1=k, \{S_t(v)\}_{t\leq \tau_2(v)} \right) \ \leq  \ \frac{c}{n^{d/2}}
    \end{align*}
\end{corollary}
\begin{proof}
Notice that if $\mathcal{A}_t$ denotes the event that the $t-$th step of the exploration process on the configuration model lead to a disagreement for the processes then $\mathbb{P}(\mathcal{A}^c_0)=0$ and for $t\geq 1$,
    \begin{align*}
        \mathbb{P}(\mathcal{A}_t^c \mid \mathcal{A}_{t-1}) \leq C  \left( \frac{t}{n}+ \frac{1}{n^d} \right) 
    \end{align*} 
     by \eqref{rate of convergence assumption simple graph} and \eqref{X-t and Z_t}.
    So by Lemma \ref{lem_tv_stoch_proc} since $\tau_2(v) \leq k_{\max}^{\ell}<\min\{n^{1/8},n^{d/2}\} $

     \begin{align*}
    &\operatorname{d_{TV}}\left(\{Z_t\}_{t\leq \tau_2(v)}\mid Z_1=k, \{S_t(v)\}_{t\leq \tau_2(v)} \right) \ \leq \ C \sum_{t=1}^{k_{\max}^{\ell}}   \left( \frac{t}{n}+ \frac{1}{n^d} \right) \\& \leq \  C\frac{ n^{\frac{1}{4}}}{ n} + C \frac{1}{n^{d/2}}.
      \end{align*} 
      It remains to notice that 
\end{proof}
Define the set of oriented edges   $\overrightarrow{E}=\{(u,v):\{u,v\}\in E(\mathbf{\Tilde{G}}_n)\}$.  For $\overrightarrow{e}\in \overrightarrow{E}$ with $\overrightarrow{e}=(u,v)$, let $\mathbf{\Tilde{G}}'_n$ be the subgraph of $\mathbf{\Tilde{G}}_n$ with the edge $\{u,v\}$ removed. Let $(\mathbf{\Tilde{G}}_n(\overrightarrow{e}))_{\ell}$  be the $\ell$-neighbourhood of $v$ within $\mathbf{\Tilde{G}}'_n$. Let $S_t(\overrightarrow{e})$ denote the number of non-backtracking walks from $v$ in  $\mathbf{\Tilde{G}}'_n$
and of length $t\leq \tau_2(v)$.
\begin{corollary}\label{cor tv for edge}
        For  $\overrightarrow{e}=(u,v)\in \overrightarrow{E}$. If $\tau_2(v)$ is the stopping time for which all vertices at distance at most $\ell$ from $v$ have been revealed during the exploration process and 
    $\operatorname{deg}(v)=k$
    for some $k\in \mathbf{d}$, then for some absolute constant $c>0$ we have:
    \begin{align}\label{bound  cor tv edge}
        \operatorname{d_{TV}}\left(\{Z_t\}_{t\leq \tau_2(v)}\mid Z_1=k-1, \{S_t(\overrightarrow{e})\}_{t\leq \tau_2(v)} \right) \ \leq \ \max\left\{\frac{1}{n^{3/4}}, \frac{1}{n^{d/2}}\right\}.
    \end{align} 
    
    Moreover $(\mathbf{\Tilde{G}}_n(\overrightarrow{e}))_{\ell}$ is a tree with probability $1-\frac{c}{n^{3/4}}$. 
\end{corollary}
   \begin{proof}
       Given Corollaries~\ref{cor tree}~and~\ref{cor total variation for vertex}, one has that $\mathbf{\Tilde{G}}_n(v)_{\ell}$ is a tree with offspring distribution $\big(\{Z_t\}_{t\leq \tau_2(v)}\ \big|\  Z_1=k \big)$,
       with probability greater than $1- {c}/{n^{3/4}}$.
       
       Working on this event, if one deletes the edge $\{u,v\}$ from $\mathbf{\Tilde{G}}(v)_{\ell}$, every  offspring of $u$ is disconnected with $v$. This implies the second statement.

The proof  of \eqref{bound  cor tv edge} is  analogous to the proof of Corollary~\ref{cor total variation for vertex}       
by considering the offspring distribution of the graph \begin{equation*}
\mathbf{\Tilde{G}}(v)_{\ell}\setminus \{w:\text{ w is an offspring of u in the graph }\mathbf{\Tilde{G}}(v)_{\ell}\}.
\qedhere
\end{equation*}
   \end{proof}
   \begin{obs}\label{observation number of edges}
       By Corollary~\ref{cor tree} and Markov's inequality it is clear that
       \begin{align*}
           &\mathbb{P}\left[\sum_{v \in [n]} \mathbf{1}\left( \mathbf{\Tilde{G}}(v)_{\ell} \text{ is not a tree} \right)\ \geq \  n^{\frac{1}{4}}\log(n)   \right]\\& \leq  \frac{n}{n^{\frac{1}{4}}\log n} \max_{v \in [n]} \mathbb{P}\left( \mathbf{\Tilde{G}}(v)_{\ell} \text{ is not a tree} \right) \ \leq \ \frac{c}{\log(n)} 
       \end{align*}
       Furthermore by Corollary~\ref{cor tv for edge}, we have a similar bound if one considers the analogous sum over all $\overrightarrow{e}\in \overrightarrow{E}$.
   \end{obs}

   We now present limiting results, i.e.\ laws of large numbers, for functionals of $S_t(\overrightarrow{e})$.
   \begin{prop}\label{law of large numbers for simple functionals}
     There exists a 
     constant $\rho>0$ such that 
     \begin{align}\label{largenumbersfor S^2}
         \frac{1}{n}\sum_{\overrightarrow{e}\in \overrightarrow{E}} \frac{\left(S_{\ell}(\overrightarrow{e})\right)^2}{a^{2\ell}} \to \rho, \text{ in probability,}
     \end{align}
     and 
     \begin{align}\label{S_2l S_l lawoflarge}
          \frac{1}{n}\sum_{\overrightarrow{e}\in \overrightarrow{E}} \frac{S_{2\ell}(\overrightarrow{e})S_{\ell}(\overrightarrow{e})}{a^{3\ell}} \to \rho, \text{ in probability,}
     \end{align}
   \end{prop}
   \begin{proof}
           We will  use \eqref{concentration Theorem}. For matchings $m\neq m'$  differing at 4 edges, let $\overrightarrow{E}(m)$ be the oriented edges of $m$ and let  $\overrightarrow{E}(m')$ be the oriented edges of $m'$.  Then for some sequence $c_n$:
     \begin{align}\label{X_l X_2l}
        \left|\sum_{\overrightarrow{e}\in \overrightarrow{E}(m)} \frac{(S_{\ell}(\overrightarrow{e}))^2}{a^{2\ell}} - \sum_{\overrightarrow{e}\in \overrightarrow{E}(m')} \frac{(S_{\ell}(\overrightarrow{e}))^2}{a^{2\ell}}\right| 
        \ \leq \ c_n
    \end{align}
    
    For any $\overrightarrow{e} \ \in \ \overrightarrow{E} (m)\cap \overrightarrow{E}(m')$,
    \begin{align*}
        |(S^{(m)}_{\ell}(\overrightarrow{e})^2-(S^{(m')}_{\ell}(\overrightarrow{e})^2| 
        \ \leq \ 
        k^{2\ell}_{\max} 
    \end{align*}
    and similarly for edges $e \in m\setminus m'$ and $e \in m'\setminus m$. Moreover the number of edges effected by the difference  between the two matching is bounded by $4 k^{\ell}_{\max}$. This implies that:
    \begin{align}\label{sum S_l(e)^2}
        \left|\sum_{\overrightarrow{e}\in \overrightarrow{E}(m)} \frac{(S_{\ell}(\overrightarrow{e}))^2}{a^{2\ell}} - \sum_{\overrightarrow{e}\in \overrightarrow{E}(m')} \frac{(S_{\ell}(\overrightarrow{e}))^2}{a^{2\ell}}\right| 
         \ \leq \ 
         c \left(\frac{k^{3\ell}_{\max}}{a^{2\ell}}\right) 
    \end{align}
    Recall that $\ell=\delta_0\log_a n$ for some $\delta_0\in \left(0,\frac{\eta d}{8}\right)$. Using that $\delta_0<\frac{\eta d}{6}$ we obtain:
    $$\left(\frac{k^{3\ell}_{\max}}{a^{2\ell}}\right)
    \ \leq \ 
    n^{\frac{1}{2}-2\delta_0}$$  
    So by~\eqref{concentration Theorem}: 
    \begin{align*}
        \mathbb{P}\left[ \left|\sum_{\overrightarrow{e}\in \overrightarrow{E}} \frac{(S_{\ell}(\overrightarrow{e}))^2}{a^{2\ell}} - \sum_{\overrightarrow{e}\in \overrightarrow{E}} \E \frac{(S_{\ell}(\overrightarrow{e}))^2}{a^{2\ell}}\right|  \geq n n^{-\delta_0}\right] 
        \\ \ \leq \ \exp \left(- C\frac{n^{1-2\delta_0}}{n^{1-4\delta_0}}\right) 
        \ = \ 
        \exp(-cn^{2\delta_0})
    \end{align*}
    So one has that the random variables in \eqref{sum S_l(e)^2} concentrate around their mean asymptotically almost surely. 
    For the mean, note that 
    if  $\chi(\overrightarrow{e})$ is the indicator that the coupling described in Corollary~\ref{cor tv for edge}  fails for any $\overrightarrow{e}\in \overrightarrow{E}$ then
    \begin{align*}
      \frac{1}{a^{2\ell}} \max\{ \E (Z_{\ell})^2\chi(\overrightarrow{e}),  \E (S_{\ell}(\overrightarrow{e}))^2 \chi(\overrightarrow{e}) \} 
      \ \leq \ 
      \frac{k^{3\ell}_{\max}}{a^{2\ell}} \max\left\{\frac{1}{n^{3/4}}, \frac{1}{n^{d/2}}\right\}.
      \ \leq \ 
       n^{-\delta_0}. 
    \end{align*} 
    So it is implied that
    \begin{align}\label{EZ=ES_l(e)}
        \frac{\E Z^2_{\ell}}{a^{2\ell}}-Cn^{-\delta_0}
        \ \leq \ 
        \sum_{\overrightarrow{e}\in \overrightarrow{E}} \frac{\E (S_{\ell}(\overrightarrow{e}))^2}{n a^{2\ell}}
        \ \leq \ 
        \frac{\E Z^2_{\ell}}{a^{2\ell}}+Cn^{-\delta_0}.
    \end{align}
    Now \eqref{sum S_l(e)^2} follows from \eqref{EZ=ES_l(e)}, Lemma~\ref{lemma limit of Zt} and Observation~\ref{obs different start}.

    For the proof of \eqref{S_2l S_l lawoflarge}, let $Y^{(u)}_{t}$ denote the offspring of a vertex $u$ at distance $t$. Let $Z^{(u)}_{t}=|Y^{(u)}_t|$ and $Z_{t}=|Y^{(o)}_t|$. Let $o$ be the root of the unimodular Galton-Watson tree, let $\mathcal{F}_t$ be the filtration generated by $(Z_1,\ldots, Z_t)$ and let $\E_{_{\mathcal{F}_t}}$ denote the conditional expectation with respect to the filtration $\mathcal{F}_t$.
    
    With the notation above, we obtain the following, whose last equality uses Observation~\ref{obs given history}:
    \begin{align}\label{EZ_lZ_2l}
    \E Z_{\ell} Z_{2\ell} 
    \ = \ 
    \E Z_{\ell}  \sum_{u \in Y^{(o)}_{\ell}} \textstyle{\E_{_{\mathcal{F}_{\ell}}}} Z^{(u)}_{\ell}
    \ = \ 
    a^{\ell} \E Z^2_{\ell}
     \end{align}

    Moreover for two different matchings $m,m'$ which differ on at most 4 half-edges,  one can show similarly to  \eqref{sum S_l(e)^2} that
the following holds  for any $\delta_0\in (0,\frac{\eta d}{10})$:
    \[        \left|\sum_{\overrightarrow{e}\in \overrightarrow{E}(m)} \frac{S_{\ell}(\overrightarrow{e})S_{2\ell}(\overrightarrow{e})}{a^{3\ell}} 
    \ - \ 
    \sum_{\overrightarrow{e}\in \overrightarrow{E}(m')} \frac{S_{\ell}(\overrightarrow{e})S_{2\ell}(\overrightarrow{e})}{a^{3\ell}}\right|  
    \ \leq \ 
    c \left(\frac{k^{5\ell}_{\max}}{a^{3\ell}}\right)
    \ \leq \ 
    c n^{\frac{1}{2}-3\delta_0} \] 

 With the above facts, the proof of \eqref{S_2l S_l lawoflarge} closely follows the proof of \eqref{largenumbersfor S^2}.
   \end{proof}
   In what follows we will use the following notation.
For any $\overrightarrow{e}_1,\overrightarrow{e}_2\in \overrightarrow{E}$ define $\overrightarrow{d}(\overrightarrow{e}_1,\overrightarrow{e}_2)$ to be the minimum natural number $m$ such that there exists a path $(\gamma_1,\gamma_2,\ldots, \gamma_m)$
such that $\overrightarrow{e}_1=(\gamma_1,\gamma_2)$ and $\overrightarrow{e}_2=(\gamma_{m-1},\gamma_{m})$ and the path is self-avoiding, i.e., it doesn't use a vertex more than once. In particular it will not backtrack.

Define $\mathcal{Y}_t(\overrightarrow{e})=\{f\in \overrightarrow{E}: \overrightarrow{d}(\overrightarrow{e},f)=t\}$. Let
\begin{align*}
      P_{\ell}(\overrightarrow{e})= \sum_{t=0}^{\ell-1} \sum_{f \in \mathcal{Y}_{t}(\overrightarrow{e})} L_t(f) 
    \end{align*}
    where 
    \begin{align*}
        L_t(f)= \sum_{g\neq h \in \mathcal{Y}_1(f) \setminus \mathcal{Y}_{t}(\overrightarrow{e}) } \Tilde{S}_{\ell-t-1}(g) \Tilde{S}_{t+1}(h)
    \end{align*}
and $\Tilde{S}_x(y)$ is the random variable $S_x(y)$ defined on $\mathbf{G}$ where all the edges of $(\mathbf{\Tilde{G}}_n(\overrightarrow{e}))_t$ have been removed for $x\in \{t,\ell-t-1\}$ and $y \in \{g,h\}$. 

\begin{obs}\label{tree P_l(e)}
If $(\mathbf{\Tilde{G}}_n(\overrightarrow{e}))_{2\ell}$ is a tree then $S_{\ell-t-1}(g)=\Tilde{S}_{\ell-t-1}(g)$ and $S_t(h)=\Tilde{S}_t(h)$, since the only link between $(\mathbf{\Tilde{G}}_n(\overrightarrow{e}))_t$ and $(\mathbf{\Tilde{G}}_n(y))_x$ is the edge $f$, which will not be taken into account by the definition of the quantities $S_x(y)$ before Corollary~\ref{cor tv for edge} for any $x\in \{t,\ell-t-1\}$ and $y\in \{g,h\}$. 
\end{obs}

In Proposition~\ref{concusion_for_simple}, the quantities $P_{\ell}(\overrightarrow{e})$ will play the role of a good approximation of $B^{\ell} (B^{*})^\ell\chi(\overrightarrow{e})$ with high probability.
We now give an asymptotic result, a law of large numbers, for $P_{\ell}(\overrightarrow{e})$.
\begin{prop}\label{lawoflarge Pl}
    There exists a constant $\rho_1>0$ such that
    \begin{align*}
        \frac{1}{n}\sum_{\overrightarrow{e}\in \overrightarrow{E}} \frac{P^2_{\ell}(\overrightarrow{e})}{a^{4\ell}}\to \rho_1, \text{ in probability.}
    \end{align*}
\end{prop}
\begin{proof}
    We will compare 
    $$F(\mathbf{\Tilde{G}}_n) = \sum_{\overrightarrow{e}\in \overrightarrow{E}} \frac{P^2_{\ell}(\overrightarrow{e})}{a^{4\ell}}$$ 
    with its mean, and then its mean with $a^{-4l}\E Q^2_{\ell}$, and then use Proposition~\ref{thm 25}. 

    As in the proof of Proposition~\ref{law of large numbers for simple functionals}, we will compare $F(m)$ and $F(m')$, for two matchings of the half-edges that differ on at most 4 half-edges.

Analogously to \eqref{sum S_l(e)^2} one can show that for any $\delta_0\in (0,\frac{\eta d}{12})$:
 \begin{align*}
     |F(m)-F(m')|
     \ \leq \
     C \frac{k^{6\ell}_{\max}}{a^{4\ell}}
     \ \leq \
     n^{\frac{1}{2}-4\delta_0}
 \end{align*}
 Applying \eqref{concentration Theorem} we see that $|\E F(\mathbf{\Tilde{G}}_n)- F(\mathbf{\Tilde{G}}_n)|$ is asymptotically negligible. 

By Corollaries \ref{cor total variation for vertex}, \ref{cor tree}, \ref{cor tree} and \eqref{prob of a tree}, for any $\overrightarrow{e}\in \overrightarrow{E}$,
 \begin{align}\label{prob tree edge 2l}
     \mathbb{P}(\mathbf{\Tilde{G}}_n(\overrightarrow{e})_{2\ell} \text{ is not a tree})
     \ \leq \ 
     \frac{c}{n^{\frac{3}{4}}} 
 \end{align}
 for any $\delta_0<\frac{d \eta}{16}$. Given \eqref{prob tree edge 2l} and Observation~\ref{tree P_l(e)} the proof continues analogously with the proof of Proposition~\ref{law of large numbers for simple functionals} and is thus omitted.
 \end{proof}
 We now prove that the growth of $S_{\ell}(\overrightarrow{e})$ is almost geometric for most $\overrightarrow{e}\in \overrightarrow{E}$.
 \begin{corollary}\label{tree and growth}
    Fix $\ell=\delta_0\log_a n$ with $\delta_0<\frac{\eta d}{8}$. For $\overrightarrow{e}\in \overrightarrow{E}$, define the event $\mathcal{E}(\overrightarrow{e})$ such that:
     \\ The graph $\mathbf{\Tilde{G}}_n(\overrightarrow{e})_{\ell}$ is a tree, and for  $0\leq t \leq \ell$ it is true that 
       for $C=C(1)$ of Proposition~\ref{growth of Zt} we have:
     \begin{align*}
          |S_t(\overrightarrow{e}) - a^{t-l}S_{\ell}(\overrightarrow{e}) | \ \leq \ C a^{\frac{t}{2}} (\log^{C} n) 
     \end{align*}
     Then for  some absolute constant $c>0$ we have:
     \begin{align*}
         \mathbb{P}\left[ \sum_{\overrightarrow{e}\in \overrightarrow{E}} \mathbf{1}\left( \mathcal{E}^c(\overrightarrow{e}) \right)\geq   n^{1- \frac{d}{2}}  \log n   \right] 
         \ \leq \ 
         \frac{c}{\log n}
     \end{align*}
     \end{corollary}
 \begin{proof}
     Each event $\mathcal{E}(\overrightarrow{e})$ holds with probability $1-\frac{c}{n^{\frac{3}{4}}}$, due to Proposition~\ref{growth of Zt} applied for $\gamma=1$ and Corollary~\ref{cor tv for edge}. 

     Thus the union bound and Markov's inequality yields:
     \begin{equation*}
          \mathbb{P}\left[ \sum_{\overrightarrow{e}\in \overrightarrow{E}} \mathbf{1}\left( \mathcal{E}^c(\overrightarrow{e}) \right)\geq n^{\frac{1}{4}} \log n \right] 
          \ \ \leq \ \ 
          c \frac{n}{n^{1-\frac{d}{2}}\log n} \cdot \frac{1}{n^{d/2}}
          \ \ \leq \ \ 
          \frac{c}{\log n} \qedhere
     \end{equation*}
     \end{proof}
     \begin{notation}\label{defn_of_mathbfE}
     In what follows   $\mathbf{E}_{\ell}$ denotes the set of edges $\overrightarrow{e}$ for which the event $\mathcal{E}(\overrightarrow{e})$ of Corollary~\ref{tree and growth} does not hold. 
        Clearly with high probability the cardinality of $\mathbf{E}_{\ell}$ is at most $n^{1-\frac{d}{2}}\log n$ by Corollary~\ref{tree and growth}. 
     \end{notation}
     Next we describe another property of $\mathbf{\Tilde{G}}_n$, which will be important later in the proof of Proposition \ref{prop:second_eigenvalue}.
     \begin{defn}\label{defn_l_tangle_free}
     Given a multigraph $G$ and a natural number $m>1$, we say  $G$ is \emph{$m$-tanglefree} if every $m$-neighbourhood of $G$ contains at most one circle.
     \end{defn}
     \begin{lemma}\label{tanglefree multi}
         With probability tending to $1$,  the graph $\mathbf{\Tilde{G}}_n$ is $\ell$-tanglefree 
         for $\ell=\delta_0 \log_a n$ and $\delta_0<\frac{\eta d}{16}$.
        \end{lemma}
        \begin{proof}
        Recall the notation of Theorem~\ref{theorem_properties_conf_model}.
            Let $v$ be a vertex of the multigraph $\Tilde{\mathbf{G}}_n$. Let $R_t(v) =\sum_{s=1}^{t\wedge\tau_1(v)}\epsilon_s$.  Given \eqref{epsilon_l_tangle_free}, for every $k \in \N$
            \begin{align}\label{1st_ineq_l_tangle_free}
                \mathbb{P}(R_{t}(v)\geq k) 
                \ \ \leq \ \ \ 
                \mathbb{P}(Y\geq k) 
            \end{align}
            where $Y$ is a binomial random variable with parameters $Ck_{\max}^{\ell}$ and $n$, where $C$ is the positive constant of \eqref{epsilon_l_tangle_free}. Moreover for binomial random variables we have:
            \begin{align}\label{2nd_ineq_l_tangle_free}
                 \mathbb{P}(Y\geq k)
                 \ \ \leq \ \ 
                 c^2\frac{k_{\max}^{2k\ell}}{n^{k}}
            \end{align}
            Thus by combining \eqref{1st_ineq_l_tangle_free}, \eqref{2nd_ineq_l_tangle_free} we conclude that:
            \begin{align*}
                \mathbb{P}(\Tilde{\mathbf{G}}_n \text{ is  $\ell$-tanglefree })
                \ \leq \ 
                \sum_{v \in V(\Tilde{\mathbf{G}}_n)}   \mathbb{P}(R_{t}(v)\geq 2)
                \ \leq \
                n   \mathbb{P}(Y\geq 2) 
                \ \leq \
              \frac{k_{\max}^{4\ell}}{n}
            \end{align*}
         Our choice of $\delta_0$ ensures  $ \mathbb{P}(\Tilde{\mathbf{G}}_n \text{ is $\ell$-tanglefree}) \ \leq \ n^{-\frac{1}{2}}$, which implies the result.
        \end{proof}
     \begin{lemma}\label{S_tx,)}
     For $\ell=\delta_0 \log_a n$ and $\delta_0\leq \frac{\eta d}{16}$,  let the $|\overrightarrow{E}|$-dimensional vector $\mathbf{S}_t= \left(S_{t}(\overrightarrow{e})\right)_{\overrightarrow{e}\in \overrightarrow{E}}$ for $t\in [\ell]$. Then with high probability, for  $1 \leq t \leq \ell-1$ the following holds for some absolute constant $c>0$:
     \begin{align*}
         \sup_{\|x\|=1:\langle \mathbf{S}_{\ell},x \rangle=0}\left| \langle \mathbf{S}_t,x \rangle\right| 
         \ \leq \ 
         a^{\frac{t}{2}} \sqrt{n}\log^c n 
     \end{align*}
     \end{lemma}
     \begin{proof}
        Fix $\|x\|=1$ with $\langle \mathbf{S}_{\ell},x \rangle=0$. Recall the set $\mathbf{E}_{\ell}$ in Notation~\ref{defn_of_mathbfE}. Then with high probability $|\mathbf{E}_{\ell}|\leq n^{1-\frac{d}{2}}\log n$. 
        We condition on the event for which this bound holds. 
        
        Since $\langle \mathbf{S}_{\ell},x \rangle=0$ we have:
        \begin{align}\label{prelim bound on S,x}
        \left| \langle \mathbf{S}_t,x \rangle\right| \leq \sum_{\overrightarrow{e}\notin \mathbf{E}_{\ell}} |x_{\overrightarrow{e}}| \left|S_t(\overrightarrow{e}) - a^{t-l}S_{\ell}(\overrightarrow{e}) \right| + \sum_{\overrightarrow{e}\in \mathbf{E}_{\ell}}a^{t-l} |x_{\overrightarrow{e}}| S_{\ell}(\overrightarrow{e})| + \sum_{\overrightarrow{e}\in \mathbf{E}_{\ell}} |x_{\overrightarrow{e}}| S_t(\overrightarrow{e})
    \end{align}
    For the second and third sums of \eqref{prelim bound on S,x} we use the Cauchy-Schwarz inequality to get
    \[ \sum_{\overrightarrow{e}\in \mathbf{E}_{\ell}}a^{t-l} |x_{\overrightarrow{e}}| S_{\ell}(\overrightarrow{e})| + \sum_{\overrightarrow{e}\in \mathbf{E}_{\ell}} |x_{\overrightarrow{e}}| S_t(\overrightarrow{e}) 
    \ \leq \
    \left(k^{t}_{\max}+\frac{k^{\ell}_{\max}}{a^{\ell-t}}\right)\sqrt{|\mathbf{E}_{\ell}|}
    \ \leq \ 
    \sqrt{n}\log^\frac{1}{2} n\]
     For the first sum of \eqref{prelim bound on S,x}, note that 
     the assumptions of Corollary~\ref{tree and growth} hold
     for $\overrightarrow{e}\notin \overrightarrow{E}$. Thus by the Cauchy-Schwarz inequality we get that 
     \[\sum_{\overrightarrow{e}\notin \mathbf{E}_{\ell}} |x_{\overrightarrow{e}}| \left|S_t(\overrightarrow{e}) - a^{t-l}S_{\ell}(\overrightarrow{e}) \right|
     \ \leq \ 
     a^{\frac{t}{2}}\log^c n \sqrt{|\overrightarrow{E}\setminus\mathbf{E}_{\ell}|} 
     \ \leq \ 
     \sqrt{n} a^{\frac{t}{2}}\log^c n\]
   for some absolute constant $c>0$.  
   The claim follows.
   \end{proof}
For $e \in \overrightarrow{E}$ with $\overrightarrow{e}=(u,v)$, let $\deg(\overrightarrow{e})=\deg(v)$, let $\overrightarrow{e}^{-1}=(v,u)$, and for $t \in \N$
   \begin{align*}
       T_t(\overrightarrow{e}) \ = \ \sum_{f \in \mathcal{Y}_t(\overrightarrow{e})} \deg(f)-1.
   \end{align*}
   Let $\mathbf{T}_t=(T_{\overrightarrow{e}^{-1}})_{\overrightarrow{e}\in \overrightarrow{E}}$. Note that the degree of each vertex is given in Definition~\ref{defn:configuration_model} so $\mathbf{T}_t$ 
   is deterministic.
   \begin{corollary}\label{cor_for_T}
   Fix $\ell=\delta_0 \log_a n$ where $\delta_0\leq \frac{\eta d}{16}$. Then with high probability, for  $1 \leq t \leq \ell-1$ we have: 
     \begin{align*}
         \sup_{\|x\|=1:\langle \mathbf{T}_{\ell},x \rangle=0}\left| \langle \mathbf{T}_t,x \rangle\right|
         \ \leq \
         a^{\frac{t}{2}} \sqrt{n}\log^c n 
     \end{align*}
     for some absolute constant $c>0$.
   \end{corollary}
\begin{proof}
    Recall the definition of the set $\mathbf{E}_{\ell+1}$ in Notation~\ref{defn_of_mathbfE}. For any $\overrightarrow{e}^{-1} \in \mathbf{E}_{\ell+1}$
    \begin{align}\label{T_l=S_l+1}
        T_{\ell}(\overrightarrow{e}^{-1})=S_{\ell+1}(\overrightarrow{e}^{-1})
    \end{align}
    Due to \eqref{T_l=S_l+1} the remainder of the proof is analogous to the proof of Lemma~\ref{S_tx,)}.
\end{proof}
   \subsection{Conclusion: Results for Simple Graphs}
   We are now ready to prove Proposition~\ref{prop:asymptotics_for_ probabilistic}. 
   The following result allows us to transfer our results from multigraphs to graphs.

   \begin{theorem}\label{thm:graphic,from multi to simple}
       Let $\mathbf{\Tilde{G}}_n$ be a sequence of uniformly chosen multigraphs satisfying the assumptions of Definition~\ref{defn:configuration_model}. Then the following hold:
       \begin{align}\label{graphic sequence}
           \liminf_n \mathbb{P}( \mathbf{\Tilde{G}}_n \text{ is simple})>0
       \end{align}
       and for any sequence of events $\mathcal{H}_n$ with 
       $
           \lim_{n \to \infty}\mathbb{P}(\mathbf{\Tilde{G}}_n \in \mathcal{H}_n)=1
       $
       we have:
       \begin{align}\label{from multi to graphs}
            \lim_{n \to \infty}\mathbb{P}(\mathbf{\Tilde{G}}_n \in \mathcal{H}_n \mid \mathbf{\Tilde{G}}_n \text{ is simple})=1
       \end{align}
   \end{theorem}
    \begin{proof}
One can find \eqref{graphic sequence} in  \cite[Thm~1.1]{janson2009probability}
and \eqref{from multi to graphs} in  \cite[Thm~2.20]{bordenave2016lecture}.

In particular \eqref{graphic sequence} implies that the sequence of degrees $d_n$ is graphic for large $n$, i.e., that there are simple graphs with $n$ vertices whose degrees are given by $d_n$. 
    \end{proof}
    \begin{obs}
        As $d_n$ is graphic for large $n$ by \eqref{graphic sequence}, one might ask whether there is a difference between the laws of $\mathbf{\Tilde{G}_n}\mid (\mathbf{\Tilde{G}_n} \text{ is simple})$ and $\mathbf{G}_n$. The answer is that there is no difference. See for example \cite[Cor~1.7]{bordenave2016lecture}.
    \end{obs}
    We will now use \eqref{from multi to graphs} to transfer our results to simple graphs. For an event $\mathcal{H}$ we use the notation 
    \begin{align*}
        \mathbb{P}_{\text{simple}}(\mathcal{H})=\mathbb{P}(\mathcal{H}\mid \mathbf{\Tilde{G}}_n \text{ is simple})
    \end{align*}
    Let $B_n$  be the non-backtracking matrix of the random graph $\mathbf{\Tilde{G}}_n\mid (\mathbf{\Tilde{G}}_n \text{ is simple })$.

     The results from Section~\ref{subsection local analysis} hold for  $\ell=\delta_0 \log_a n$ where $\delta_0<\frac{\eta d}{16}$. So we will work for $\ell$ of that order.

    Let the $\overrightarrow{E}$-dimensional vector
   \begin{align*}
       \psi= (\psi_{\overrightarrow{e}})_{\overrightarrow{e} \in \overrightarrow{E}} 
   \end{align*}
be   such that if $\overrightarrow{e}=(u,v)$ then $\psi_{\overrightarrow{e}}=\deg(v)-1$. Let $\Tilde{\psi}$  be the vector whose entry at $\overrightarrow{e}$ is $\psi_{\overrightarrow{e}^{-1}}$.
\begin{obs}\label{obs_for__gen_non_back_}
    The following property for the non-backtracking matrix can be easily proven: Let $z_1,z_2 \in \R^{|\Delta^{(n)}|}$ and $k \in \N$
    \begin{align*}
        \langle z_1, B^k_n z_2 \rangle= \langle (B^{*}_n)^{k} \Tilde{z_1},\Tilde{z_2}\rangle
    \end{align*}
    where $\Tilde{z}_i$ is the vector such that $(\Tilde{z}_{i})_{\overrightarrow{\epsilon}}=(z_{i})_{{\overrightarrow{\epsilon}^{-1}}}$ for any $\overrightarrow{e} \in \overrightarrow{E}$.
\end{obs}
    \begin{prop}\label{concusion_for_simple}
     Let $\ell=\delta_0 \log_a n$ with $\delta_0<\frac{\eta d}{16}$. The following holds
      for some absolute constant $c>0$  and $0\leq t\leq \ell-1$:
      \begin{enumerate}
        \item\label{tanglefree simple} $\lim_{n \to \infty }\mathbb{P}_{\text{simple}}(\mathbf{\Tilde{G}}_n \text{ is $\ell$-tanglefree}) \ = \ 1$ 
\\
        \item\label{B^lx,} 
        $\displaystyle{\lim_{n\to \infty}\mathbb{P}_{\text{simple}}\left(   \sup_{\|x\|=1:\langle (B^{*}_n)^\ell\Tilde{\psi},x \rangle=0}\left| \langle (B^*_n)^t\Tilde{\psi},x \rangle\right|\leq a^{\frac{t}{2}} \sqrt{n}\log^c n  \right)=1
        }$
\\
        \item There exist constants $c_1>c_0>0$ such that
        \begin{align}\label{B^2l/B^l}
            \lim_{n \to \infty}\mathbb{P}_{\text{simple}}\left(c_0 a^{\ell}\leq  \frac{\|B^{\ell}_n(B^*_n)^\ell\Tilde{\psi}\|}{\|B^{\ell}_n \psi\|}\leq c_1 a^{\ell} \right)=1
        \end{align}
        and 
        \begin{align}\label{<B^l,B^2l>}
            \lim_{n \to \infty} \mathbb{P}_{\text{simple}}\left( \frac{\left|\langle B^{\ell}_n(B^{*}_n)^{\ell} \Tilde{\psi}, (B^*_n)^\ell\Tilde{\psi}\rangle\right|}{\|B^{\ell}_n(B^{*}_n)^{\ell} \Tilde{\psi}\|\|(B^*_n)^\ell\Tilde{\psi}\| } \geq c_0\right)=1.
        \end{align}
      \end{enumerate}
    \end{prop}
    
        \begin{proof}[Proof of \eqref{tanglefree simple}]
            This is a direct consequence of \eqref{from multi to graphs} and Lemma~\ref{tanglefree multi}.
        \end{proof} 
        \begin{proof}[Proof of \eqref{B^lx,}]
           Recall the notation from Corollary~\ref{cor_for_T}. It suffices to observe that
            \begin{align}\label{observation on }
              (B^*_n)^t\Tilde{\psi}=\mathbf{T}_t
            \end{align}
            for any $t\in [\ell]$. So \eqref{B^lx,} is a direct consequence of \eqref{from multi to graphs} and Corollary~\ref{cor_for_T}. 
                \end{proof}
        \begin{proof}[Proof of \eqref{B^2l/B^l}]
        By \ref{observation on } we have 
        \begin{align*}
            \|(B^*_n)^\ell\Tilde{\psi}\|^2=\sum_{\overrightarrow{e}\in \overrightarrow{E}} \left(T_{\ell}(\overrightarrow{e}^{-1})\right)^2= \sum_{\overrightarrow{e}\in \overrightarrow{E}} \left(T_{\ell}(\overrightarrow{e})\right)^2
        \end{align*}  
           Furthermore by \eqref{from multi to graphs} and Observation~\ref{observation number of edges}, we have: 
      \begin{align*}
          \lim_{n\to \infty}\mathbb{P}_{\text{simple}}\left[\sum_{\overrightarrow{e} \in \overrightarrow{E}} \mathbf{1}\left( \mathbf{\Tilde{G}}(\overrightarrow{e})_{2\ell+1} \text{ is not a tree} \right)\geq n^{\frac{1}{4}}\log(n)   \right] \ \ = \ \ 0
      \end{align*}
      Let $\mathcal{E}_{2\ell+1}$ be the set of edges for which the graph $\mathbf{\Tilde{G}}_n(\overrightarrow{e})_{2\ell+1}\mid (\mathbf{\Tilde{G}}_n \text{ is simple })$ is not a tree. As mentioned in \eqref{T_l=S_l+1}, for any $\overrightarrow{e}$ we have: 
      \begin{align*}
          T_{\ell}(\overrightarrow{e})=S_{\ell+1}(\overrightarrow{e})
      \end{align*}
      Thus,
      \begin{align*}
          \left| \sum_{\overrightarrow{e}\in \overrightarrow{E}} (T_{\ell}(\overrightarrow{e}))^2- (S_{\ell+1}(\overrightarrow{e}))^2 \right| 
        \ \leq \
          k^{2(\ell+1)}_{\max} n^{\frac{1}{4}} \log(n) 
        \  \leq \
          k^2_{\max}n^{\frac{1}{4}} n^{\frac{1}{8}} \log n= o(a^{2\ell}n).
      \end{align*}
        Thus, we can apply \eqref{largenumbersfor S^2} and \eqref{from multi to graphs} to get that the event 
        \begin{align}\label{B^l=a^l}
         \frac{c \rho}{2} \sqrt{n}a^{\ell}  
         \ \leq \  
         \|  (B^*_n)^t\Tilde{\psi}\| 
         \ \leq \ 
         C 2\rho \sqrt{n}a^{\ell} 
        \end{align}
        holds with probability tending to $1$, for some absolute constants $c,C>0$.

      Observe that for $\overrightarrow{e}\notin \mathcal{E}_{2\ell+1}$ with $\overrightarrow{e}= (u,v)$, we have:
      \begin{align*}
           B^\ell_n (B_n^{*})^\ell\Tilde{\psi}(\overrightarrow{e})= P_{\ell}(\overrightarrow{e})+(\deg(v)-1) S_{\ell}(\overrightarrow{e})
      \end{align*}
      Let 
      \begin{align*}
         \Tilde{S}_{\ell}(\overrightarrow{e})= (\deg(v)-1) S_{\ell}(\overrightarrow{e})
      \end{align*}
      Thus, we get that
      \begin{align}\label{Bl-P_l}
          \| B^{\ell}_n (B_n^{*})^\ell\Tilde{\psi} - P_{\ell}- \Tilde{S}_{\ell}\|\leq k^{4\ell}_{\max} |\mathcal{E}_{2\ell}|\leq n^{\frac{1}{2}}\log n=o(\sqrt{n}a^{2\ell})
      \end{align}
      Thus by Proposition~\eqref{Bl-P_l}, Proposition~\ref{lawoflarge Pl}, \eqref{largenumbersfor S^2} and \eqref{from multi to graphs}, there are two constants $c_2,c_3>0$ such that for large enough $n$
      
        \begin{align}\label{BlB*l=a^2l}
c_3a^{2\ell}
\sqrt{n} \ \leq \ \|B^{\ell}_n(B^*_n)^\ell\Tilde{\psi}\| \ \leq \ c_2 \sqrt{n}a^{2\ell}
        \end{align}
        with probability tending to $1$.

        Now \eqref{B^2l/B^l} follows from \eqref{B^l=a^l} and \eqref{BlB*l=a^2l}. 
        \end{proof}
        \begin{proof}[Proof of \eqref{<B^l,B^2l>}]
        By Observation~\ref{obs_for__gen_non_back_} and \eqref{observation on },
        \[\left|\langle B^{\ell}_n(B^{*}_n)^{\ell} \Tilde{\psi}, (B^*_n)^\ell\Tilde{\psi}\rangle\right|=\left|\langle (B^*_n)^\ell\Tilde{\psi}, (B_n^*)^{2\ell}\Tilde{\psi}\rangle\right|=\left|\langle \mathbf{T}_{\ell}, \mathbf{T}_{2\ell}\rangle\right|=\sum_{\overrightarrow{e}\in \overrightarrow{E}} T_{2\ell}(\overrightarrow{e})T_{\ell}(\overrightarrow{e}).\]

    For any $\overrightarrow{e} \notin \mathcal{E}_{2\ell+1} $ 
          \begin{align*}
              T_{2\ell}(\overrightarrow{e})T_{\ell}(\overrightarrow{e})= S_{2\ell+2}(\overrightarrow{e}) S_{\ell+1}(\overrightarrow{e}).
          \end{align*}
          Hence
          \begin{align*}
              \left| \sum_{\overrightarrow{e} \in \overrightarrow{E} } T_{2\ell}(\overrightarrow{e})T_{\ell}(\overrightarrow{e})- S_{2\ell+2}(\overrightarrow{e}) S_{\ell+1}(\overrightarrow{e}) \right|
              \ \leq \ 
              2k^3_{\max} k^{3\ell}_{\max} n^{\frac{1}{4}}\log (n) \ = \ o(a^{3\ell}n).
          \end{align*}
           Thus, \eqref{<B^l,B^2l>} follows from \eqref{S_2l S_l lawoflarge}, \eqref{B^l=a^l} and \eqref{BlB*l=a^2l}. 
        \end{proof}
\section{Matrix expansion}\label{sectionmatriexpansion}
In this section we establish a bound for the ``second eigenvalue'' of the configuration model, i.e.\ Proposition~\ref{prop:second_eigenvalue} which is an analogue of  \cite[Prop~11]{bordenave2015non} for our set of matrices. Our main influence for this part is the proof of  \cite[Thm~2]{bordenave2015new}. There, the second eigenvalue of random $d$-regular graphs is bounded with high probability, when these graphs are viewed as a specific case of the configuration model. We generalize this result to our model, The arguments used in \cite{bordenave2015new} are completely adaptable to our set of graphs, so we will omit the proofs of a few technical lemmas which
are completely analogous to their counterparts in  \cite{bordenave2015new}.

In this section we define the non-backtracking matrix in terms of the half edges of the multigraph. Our definition is equivalent to the one given in Definition~\ref{defn_of_non-backtracking}.

\begin{defn}\label{denf_non_back_conf}
Let $\Tilde{\mathbf{G}}_n$ be the configuration model and let $\Delta^{(n)}$ be the set of half edges of $\Tilde{\mathbf{G}}_n$. For  $\sigma \in M(\Delta^{(n)})$ define the following $|\Delta^{(n)}| \times |\Delta^{(n)}|$ symmetric matrices:
\begin{align*}
   M_{f,e} \ = \ M_{e,f}\ = \ \mathbf{1}\left( \sigma(e)=f \right), \text{ for } e,f \in \Delta^{(n)}
\end{align*}
and if $e=(u,i)$ and $f=(v,j)$
\begin{align*}
    N_{f,e}\ =\ N_{e,f}\ =\ \mathbf{1}\left( u=v, i\neq j \right)  
\end{align*}
The non-backtracking matrix of  $\Tilde{\mathbf{G}}_n$ is defined as 
\begin{align*}
    B_n=M_n\cdot N_n
\end{align*}
\end{defn}
Clearly the matrix $M$ captures the randomness in the non-backtracking matrix and the matrix $N$ captures the non-backtracking properties of $B$.

In what follows let $N=|\Delta^{(n)}|=\sum_i n_ik_i$ and 
let $\chi$ be the $N$-dimensional vector with all entries $1$.
Let $\Tilde{\psi}$  be the $N$-dimensional vector with $\Tilde{\psi}_{e}= \deg(u)-1$
for $e=(u,i) \in \Delta^{(n)}$.
In this subsection we will prove 
Proposition~\ref{prop:second_eigenvalue} by analyzing the non backtracking matrix of the configuration model and then by using some bounds from Proposition \ref{concusion_for_simple} and Theorem \ref{thm:graphic,from multi to simple}. 

Below we  omit the dependence of the matrices of Definition~\ref{denf_non_back_conf} on the parameter~$n$.

For $e,f \in \Delta^{(n)}$ and $k \in [\ell]$
\begin{align*}
    B^{k}_{e,f}= \sum_{\gamma \in \Gamma^{k}_{e,f}} \prod_{s=1}^k M_{\gamma_{2s-1},\gamma_{2s}} 
\end{align*}
where $\Gamma_{e,f}^k$ is the set of non-backtracking walks $(\gamma_1,\ldots, \gamma_{2k+1})$, where $\gamma_{t}=(v_t,i_t)$ are half-edges with $\gamma_1=e$ and $\gamma_{2k+1}=f$, and where $v_{2t}=v_{2t+1}$ and $\gamma_{2t}\neq \gamma_{2t+1}$ for  $t\geq 1$. 

For  $e,f \in \Delta^{(n)}$ and $k\in [\ell]$ define the matrix $B^{(k)}$ as follows:
\begin{align*}
    B^{(k)}_{e,f}=\sum_{\gamma \in F^{k}_{e,f}} \prod_{s=1}^{\ell} M_{\gamma_{2s-1},\gamma_{2s}} 
\end{align*}
where $F_{e,f}^k$ is the set of tanglefree paths subset of $\Gamma_{e,f}^{k}$. Let $F^{k}$ be the set of tanglefree paths of length $k$.  Clearly if $\Tilde{\mathbf{G}}_n$ is $\ell$-tanglefree then 
\begin{align*}
    B^{\ell}=B^{(\ell)}
\end{align*}

The centered version of the matrix $B^{(k)}$ is defined as follows:
For $e,f \in \Delta^{(n)}$ let
\begin{align*}
    \underline{M}_{e,f} \ = \ M_{e,f}-\frac{1}{N}
\end{align*}
and for $k\in [\ell]$ define
\begin{align*}
    \underline{B}^{(k)}_{e,f} \ = \ \sum_{\gamma \in F^{k}_{e,f}} \prod_{s=1}^{\ell} \underline{M}_{\gamma_{2s-1},\gamma_{2s}}
\end{align*}
Fix $k \in [\ell]$. Let $F_{k,ef}$ denote the set of walks $\gamma=(\gamma_1,\ldots, \gamma_{2\ell+1})$ that  decompose into three walks $(\gamma',\gamma'',\gamma''')$ with $\gamma' \in F^{k-1}$, $\gamma'' \in F^{1}$ and $\gamma''' \in F^{\ell-k}$ where $\gamma'=(\gamma_1,\ldots, \gamma_{2k-1})$, $\gamma''=(\gamma_{2k-1},\gamma_{2k},\gamma_{2k+1})$ and $\gamma'''= (\gamma_{2k+1},\ldots, \gamma_{2\ell+1})$. Clearly $F^{\ell}_{e,f} \subseteq F^k_{k,ef}$ but the converse may not hold, see the proof of \cite[Lem~7]{bordenave2015new}.

We shall now express the entries of  $B^{(\ell)}$ in terms of the entries of $\underline{B}^{(\ell)}$ and the entries of lower order powers of $B$.

\begin{lemma}
Letting $\Tilde{\psi}^*$ be the transpose of the vector $\Tilde{\psi}$ we have:
    \begin{align}\label{matrix_identity_for_B^(l)}
        B^{(\ell)}= \underline{B}^{(\ell)}+ \frac{1}{N}\sum_{k=1}^{\ell}\underline{B}^{(k-1)}\chi \Tilde{\psi}^* B^{(\ell-k)}-  \frac{1}{N}\sum_{k=1}^{\ell}R_k^{(\ell)} 
     \end{align}
And for $k \in [\ell]$, and following the convention that 
$\prod_{\emptyset}=1$ we have:
     \begin{align*}
        (R_k^{(\ell)})_{e,f}= \sum_{\gamma \in F_{k,ef}\setminus F_{e,f}^{\ell}} \prod_{s=1}^{k-1} \underline{M}_{\gamma_{2s-1},\gamma_{2s}
        } \prod_{s=k+1}^{\ell}M_{\gamma_{2s-1},\gamma_{2s}}
    \end{align*}
\end{lemma}
\begin{proof}
    The statement follows from \eqref{1stidentityforBl}  and \eqref{2dnidentityforBl} below:
    By the identity 
    \begin{align*}
        \prod_{i=1}^{\ell} x_i \ = \ \prod_{i=1}^{\ell}y_i+
        \sum_{k=1}^{\ell} \prod_{i=1}^{k-1} y_k (x_k-y_k)\prod_{i=k+1}^{\ell} x_i
    \end{align*}
    we see that for $e,f\in \Delta^{(n)}$ 
    \begin{align}\label{1stidentityforBl}
    B^{(\ell)}_{e,f} 
    \ = \ 
    \underline{B}^{(\ell)}_{e,f}+\sum_{k=1}^{\ell} \sum_{\gamma \in F^{\ell}_{e,f}}\prod_{s=1}^{k-1} \underline{M}_{\gamma_{2s-1},\gamma_{2s}
        } \frac{1}{N} \prod_{s=k+1}^{\ell}M_{\gamma_{2s-1},\gamma_{2s}}.
    \end{align}
Fix $k\in \ell$ and $e,f \in \Delta^{(n)}$ and recall the definition of the paths $F_{k,ef}$. Then
\begin{multline}\label{2dnidentityforBl}
    \sum_{\gamma \in F_{k,ef}}\prod_{s=1}^{k-1} \underline{M}_{\gamma_{2s-1},\gamma_{2s}
        } \frac{1}{N} \prod_{s=k+1}^{\ell}M_{\gamma_{2s-1},\gamma_{2s}}
        \\= 
        \sum_{a,b \in \Delta^{(n)}} \sum_{\gamma' \in F^{k-1}_{e,b}} \prod_{s=1}^{k-1} \underline{M}_{\gamma'_{2s-1},\gamma'_{2s}
        } \sum_{\gamma'' \in F^1_{a,b}}  \frac{1}{N} \sum_{\gamma''' \in F_{b,f}^{\ell-k}} \prod_{s=1}^{\ell-k}M_{\gamma'''_{2s-1},\gamma'''_{2s}}
        \\= 
        \frac{1}{N} \sum_{a,b \in \Delta^{(n)}} \underline{B}^{(k-1)}_{e,a} (\deg(b)-1) B^{(\ell-k)}_{b,f}   
\end{multline}
    where the first equality of \eqref{2dnidentityforBl} uses the definition of the paths in $F_{k,ef}$ and the second equality  uses that $F_{a,b}^{1}$ contains exactly $\deg(b)-1$ possible paths by our definition of the paths in $\Gamma^{k}_{e,f}$ for  $k \in [\ell]$. 
\qedhere
\end{proof}

The proof of Proposition~\ref{prop:second_eigenvalue} relies on the behavior of the entries of the matrix $B^{k}$ on large walks. We now give  definitions concerning such large walks.

\begin{defn}\label{defn_walks_half_edges}
    Let $\mathcal{E}^k$ denote the set of walks $\gamma=(\gamma_1,\gamma_2, \ldots, \gamma_k)$, where $\gamma_i=(v_i,t_i)$ is a half-edge of $\mathbf{\Tilde{G}_n}$.
    For  $\gamma \in \mathcal{E}^k$ define 
    \begin{itemize}
    \item $V_\gamma= \{v_i,  i\in[k]\}$. So  $V_\gamma$ is the set of vertices of $\mathbf{\Tilde{G}_n}$ that participate via a half-edge on $\gamma$.
    \item $E_\gamma= \left\{ \{\gamma_i,\gamma_{i+1}\}, i \in [k]\right\}$, i.e.\ $E_\gamma$ is the set of pairs of half-edges that are connected in the walk $\gamma$. We think of $\gamma$ as a multigraph with vertex set $V_{\gamma}$ and edge set $E_{\gamma}$.
    \item For $e \in \Delta^{(n)}$, the multiplicity of $e$ in $\gamma$ is denoted $m_{\gamma}(e)=\sum_{i=1}^k \mathbf{1}(e=\gamma_i)$. 
    \item For an edge $\{e,f\}$ in $\gamma$
   we denote its multiplicity
by $$m_{\gamma}(\{e,f\})=\sum_{i=1}^{k-1}\mathbf{1}\left( \{e,f\}=\{\gamma_i,\gamma_{i+1}\} \right).$$
           \item An edge $\{e,f\} \in E_{\gamma}$ is  \emph{consistent} if $m_{\gamma}(e)=m_{\gamma}(f)=m_{\gamma}(\{e,f\})$.\\ Otherwise it is  \emph{inconsistent}. 
    \end{itemize}
\end{defn}

Recall the Pochhammer
symbol for non-negative integers $n,k$:
\begin{align*}
    (n)_k=\prod_{i=0}^{k-1}(n-i).
\end{align*}

The following Proposition is crucial to our later analysis.

\begin{prop}\label{prop_for_mean_value_path}
   There exists some universal constant $c>0$ such that for every walk $\gamma \in \mathcal{E}^{2k}$, $k<c\sqrt{N}$ and $k_0 \in [k]$
   \begin{align}\label{Prop_magni_its_path}
      \left| \E \prod_{t=1}^{k_0} \underline{M}_{\gamma_{2t-1,2t}} \prod_{t=k_0+1}^{k} M_{\gamma_{2t-1,2t}} \right| 
      \ \leq \ 2^{b} \left(\frac{1}{N}\right)^{d} \left( \frac{3k}{\sqrt{N}}\right)^{a_1}
   \end{align}
where $d=|E_{\gamma}|$, where $b$ is the number of $t\in [k_0]$ with $\{\gamma_{2t-1},\gamma_{2t}\}$ inconsistent, and $a_1$ is the number of $t\in [k_0]$ with $\{\gamma_{2t-1},\gamma_{2t}\}$ consistent. 
\end{prop}
\begin{proof}
    This is an analogue of \cite[Prop~11]{bordenave2015new} and its proof is completely analogous.
 \end{proof}
 Given Proposition \ref{prop_for_mean_value_path} we can 
 analyze $\underline{B}^{\ell}$.  
\subsection{Application of the trace method}
Firstly we will bound the matrix $\underline{B}^{(\ell)}$ using the trace method. Specifically under the convention that $e_{2m+1}=e_{2m}$ we have that for any $m \in \N$,
\begin{multline}\label{trace_of_centered_B}
    \| \underline{B}^{(\ell)}\|^{2m} \ = \
    \| (\underline{B}^{(\ell)})^{2m}\|
    \ = \
    \left\| \left(\underline{B}^{(\ell)} (\underline{B}^{(\ell)})^{*}\right)^{m} \right\| 
    \ \leq \  
    \operatorname{trace} \left\{\left( \underline{B}^{(\ell)} (\underline{B}^{(\ell)})^{*} \right)^m \right\}\\ 
    \ \leq \ 
    \sum_{e_1,e_2, \ldots, e_{2m}} \prod_{t=1}^m (\underline{B}^{(\ell)})_{e_{2i-1},e_{2i}} \underline{B}^{(\ell)}_{e_{2i+1},e_{2i}}
    \ = \ 
    \sum_{\gamma= (\gamma_{1},\gamma_{2}, \ldots \gamma_{2m})}\prod_{i=1}^{2m} \prod_{t=1}^{\ell} \underline{M}_{\gamma_{i,2t-1}, \gamma_{i,2t}}  
\end{multline}
where in the right hand-side of the last equality in \eqref{trace_of_centered_B}, the sum is taken over all paths $(\gamma_1,\gamma_2, \ldots, \gamma_{2m})$ with $\gamma_i=(\gamma_{i,1},\gamma_{i,2}, \ldots, \gamma_{i,2\ell+1})$ and $\gamma_i \in F^{\ell}$, i.e.\ non-backtracking tanglefree paths, and for $i \in [m]$
\begin{align*}
   \gamma_{2i,1}= \gamma_{2i+1,1} \text{ \ and \ } \gamma_{2i-1,2\ell+1}= \gamma_{2i,2\ell+1}
\end{align*} 
under the convention that $\gamma_{2m+1}=\gamma_1$.

Note that the last term of each path, i.e.\ $\gamma_{i,2\ell+1}$ with $i \in [2m]$, does not participate in the product \eqref{trace_of_centered_B}. Let $\gamma_{i,t}=(v_t,j_t)$. Given $v_{i,2\ell+1}$ with $i \in [2m]$, there are fewer than $k_{\max}^m$ paths $\gamma$ having the same contribution in \eqref{trace_of_centered_B}. Thus after performing a change of variables $\gamma'_{2i,t}=\gamma_{2i,2\ell+1-t}$ and since $\underline{M}$ is symmetric, we have the bound
\begin{align*}
\left|    \sum_{\gamma= (\gamma_{1},\gamma_{2}, \ldots, \gamma_{2m})}\prod_{i=1}^{2m} \prod_{t=1}^{\ell} \underline{M}_{\gamma_{i,2t-1}, \gamma_{i,2t}} \right|   \ \leq \  k^m_{\max} \left|\sum_{\gamma \in W_{\ell ,m}} \prod_{i=1}^{2m} \prod_{t=1}^{\ell} \underline{M}_{\gamma_{i,2t-1}, \gamma_{i,2t}}\right|
\end{align*}
where for $k \in \N$, we let $W_{k,m}$ be the set of walks in $\mathcal{E}^{2k\times 2m}$ with $\gamma=(\gamma_1,\gamma_2, \ldots, \gamma_{2m})$ where $\gamma_i=(\gamma_{i,1},\gamma_{i,2},\ldots, \gamma_{i,2k+1})$ and $\gamma_i \in F^{\ell}$, and for $i \in [m]$
\begin{align*}
   v_{2i,1}= v_{2i-1,1} \text{ \ and \ } \gamma_{2i+1,2k+1}= \gamma_{2i,2k}
\end{align*} 
where $\gamma_{2m+1}=\gamma_{1}$ and $\gamma_{i,t}=(v_{i,t},j_{i,t})$.

By Hölder's inequality, we have the following bound:
\begin{align*}
  (\E \| \underline{B}^{(\ell)}\|)^{2m} \ \leq \  \E \| \underline{B}^{(\ell)}\|^{2m} \ \leq \ k^m_{\max} \sum_{\gamma \in W_{\ell ,m}} \left|\E \prod_{i=1}^{2m} \prod_{t=1}^{\ell} \underline{M}_{\gamma_{i,2t-1}, \gamma_{i,2t}}\right|
\end{align*}
To bound $\| \underline{B}^{(\ell)}\|$ we will need a better understanding of the walks in $W_{\ell ,m}$ which we now partition as follows.

\subsection{Path analysis} 
Recall that $V_{\gamma}$ and $E_{\gamma}$ are the vertices and edges of a walk as in Definition~\ref{defn_walks_half_edges}. For $d,s \in \N$, let $W_{\ell ,m}(s,d)$ be the set of $\gamma \in W_{\ell ,m}$ with $|V_{\gamma}|= s$ and $|E_{\gamma}|=d$.

For $\gamma \in W_{\ell ,m}(s,d)$ there is a natural ordering of the vertices $v\in V_{\gamma}$. Specifically if $s\geq 2$ then for $v_1,v_2 \in V_{\gamma}$, we declare 
\begin{align}\label{ordering_in_paths}
    v_1 <v_2
\end{align} 
if letting $\gamma_{i,t}=(v_1,j_t)$ and $\gamma_{i',t'}=(v_2,j_{t'})$ be the first half-edges in $\gamma$ adjacent to $v_1$ and $v_2$, either $i<i'$ or else $i=i'$ and $t<t'$.  

Based on this ordering we shall further partition $ W_{\ell ,m}(s,d)$.
Recall the configuration model of Definition~\ref{defn:configuration_model}. For any set of non-negative integers $\{s_{i}\}_{i=k_{\min}}^{k_{\max}}$ with $\sum_{i=k_{\min}}^{k_{\max}}s_i=s$ and any family $\{S_{i}\}_{i=k_{\min}}^{k_{\max}}$ of disjoint subsets of $[s]$ satisfying $[s]=\cup_{i=k_{\min}}^{k_{\max}}S_i$ and $|S_i|=s_{i}$, we define the set of paths
\begin{align*}
    W_{\ell ,m}(s,d, \{s_{i}\}_{i=k_{\min}}^{k_{\max}}, \{S_{i}\}_{i=k_{\min}}^{k_{\max}})
\end{align*}
such that if $\gamma \in  W_{\ell ,m}(s,d, \{s_{i}\}_{i=k_{\min}}^{k_{\max}}, \{S_{i}\}_{i=k_{\min}}^{k_{\max}})$ and if $(v_1,v_2,\ldots, v_{s})$ are the vertices in $V_{\gamma}$ ordered as in \eqref{ordering_in_paths}, then for $i \in \{k_{\min},\ldots, k_{\max}\}$ and $r \in S_i$ the vertex $v_r$ has degree $i$.

For $i \in \{k_{\min},\ldots, k_{\max}\}$, let $n_i$ be the number of vertices with degree~$i$ as in Definition~\ref{defn:configuration_model}. Let $V(\mathbf{ \Tilde{G}}_n)_i$ be the set of vertices with degree~$i$.

Two walks $\gamma, \gamma' \in W_{\ell ,m}(s,d, \{s_{i}\}_{i=k_{\min}}^{k_{\max}}, \{S_{i}\}_{i=k_{\min}}^{k_{\max}})$ are \emph{equivalent}, denoted by $\gamma \sim \gamma'$, if there are permutations $\beta_i$ of the sets $V(\mathbf{ \Tilde{G}}_n)_i$ and $(c_1^{(i)}, \ldots, c^{(i)}_{n_i})$ of $[i]^{n_i}$ such that for $(r,t) \in [2m]\times [2l]$ if $\gamma_{r,t}=(v_{r,t},j_{r,t})$  and $\gamma'_{r,t}=(v_{r,t}',j_{r,t}')$ then if $v_{r,t} \in V(\mathbf{ \Tilde{G}}_n)_i$ we have $v'_{r,t}\in V(\mathbf{ \Tilde{G}}_n)_i$, $\beta_i(v'_{r,t})=v_{r,t} $ and $c_{v_{r,t}}(j_{v_{r,t}})=j'_{v_{r,t}}$. 

Let $\mathcal{W}_{\ell ,m}(s,d, \{s_{i}\}_{i=k_{\min}}^{k_{\max}}, \{S_{i}\}_{i=k_{\min}}^{k_{\max}})$  be the set of equivalence classes of 
\\ $W_{\ell ,m}(s,d, \{s_{i}\}_{i=k_{\min}}^{k_{\max}}, \{S_{i}\}_{i=k_{\min}}^{k_{\max}})$. 

For $\gamma \in W_{\ell ,m}$ let 
\begin{align*}
   \mu(\gamma)= \E \prod_{i=1}^{2m} \prod_{t=1}^{\ell} \underline{M}_{\gamma_{i,2t-1}, \gamma_{i,2t}} 
\end{align*}
If $\gamma \sim \gamma'$ then
 $   \mu(\gamma)=\mu(\gamma')$.
The following bounds $\mu(\gamma)$:
\begin{lemma}\label{Lemma_m(g)_bound}
    Let $\gamma \in W_{\ell ,m}$. Suppose ${\ell}m=o(\sqrt{N})$ and $|V_{\gamma}|=s$ and $|E_{\gamma}|= d$. Then
    \begin{align*}
        | \mu (\gamma)| \leq c^{g+m} \left( \frac{1}{N} \right)^d \left( \frac{(2\ell m)^2}{N} \right)^{(d-2g-(\ell+2)m)_+}
    \end{align*}
    where $g=d-s+1$, and $c$ is an absolute constant, and    $x_+=\max\{x,0\}$ for $x \in \R$.
\end{lemma}
\begin{proof}
    The proof follows from Proposition~\ref{prop_for_mean_value_path} and is  analogous to the proof of \cite[Lem~17]{bordenave2015new} and thus omitted. 
\end{proof}

We now bound the cardinality of $\mathcal{W}_{\ell ,m}(s,d, \{s_{i}\}_{i=k_{\min}}^{k_{\max}}, \{S_{i}\}_{i=k_{\min}}^{k_{\max}})$.

\begin{lemma}\label{Lemma_bound_on_eq_classes}
    Fix an equivalence class among the paths $\mathcal{W}_{\ell ,m}(s,d, \{s_{i}\}_{i=k_{\min}}^{k_{\max}}, \{S_{i}\}_{i=k_{\min}}^{k_{\max}})$.  
    Let $g=d-s+1$. If $g<0$ then $\mathcal{W}_{\ell ,m}(s,d, \{s_{i}\}_{i=k_{\min}}^{k_{\max}}, \{S_{i}\}_{i=k_{\min}}^{k_{\max}})$ is empty. 
    Otherwise
    \begin{align*}
        \left| \mathcal{W}_{\ell ,m}(s,d, \{s_{i}\}_{i=k_{\min}}^{k_{\max}}, \{S_{i}\}_{i=k_{\min}}^{k_{\max}}) \right| \ \leq \ (4\ell m )^{6mg+6m}.
    \end{align*}
    \end{lemma}
\begin{proof}
    The proof is analogous to the proof of \cite[Lem~16]{bordenave2015new} and thus omitted. 
\end{proof}
The following bounds the number of elements in an equivalence class:
\begin{lemma}\label{Lemma_cardinality_of_eq_class}
    For $\gamma \in W_{\ell ,m}(s,d, \{s_{i}\}_{i=k_{\min}}^{k_{\max}}, \{S_{i}\}_{i=k_{\min}}^{k_{\max}})$
the number of paths equivalent to $\gamma$ is bounded by the following, where $g=d-s+1$:
    \begin{align*}
        \prod_{i =k_{\min}}^{k_{\max}} (i (i-1) n_i)^{s_i} (k_{\max}-1)^{2g-1}
    \end{align*}
\end{lemma}
\begin{proof}
    This is analogous to the proof of \cite[Lem~15]{bordenave2015new}.
\end{proof}

We now prove the bound on $\underline{B}$.

    Recall the parameter $a$ of Theorem~\ref{thm:probabilistic_claim}. 
    By the  multinomial theorem 
    \begin{align*}
        \left(\sum_{i=k_{\min}}^{k_{\max}} i(i-1)n_i\right)^s = \sum_{s_{k_{\min}, \ldots, k_{\max}\geq 0 : \sum_{i}s_i=s}}
        \binom{s}{s_{k_{\min}},\ldots, s_{k_{\max}} } 
        \prod_{i=k_{\min}}^{k_{\max}}  (i(i-1)n_i)^{s_i}.
    \end{align*}

It is implied that 
\begin{align*}
    (a+o(1))^s= N^{-s} \sum_{s_{k_{\min}, \ldots, k_{\max}\geq 0 : \sum_{i}s_i=s}}
    \binom{s}{s_{k_{\min}},\ldots, s_{k_{\max}} } 
    \prod_{i=k_{\min}}^{k_{\max}}  (i(i-1)n_i)^{s_i}
\end{align*}
Thus by combining Lemmas~\ref{Lemma_bound_on_eq_classes}~\ref{Lemma_cardinality_of_eq_class}~and~\ref{Lemma_m(g)_bound} we conclude that for $m= \lfloor\frac{\log n}{13\log \log n}\rfloor$  
\begin{multline*}
   \sum_{\gamma \in W_{\ell ,m}} \left|\E \prod_{i=1}^{2m} \prod_{t=1}^{\ell} \underline{M}_{\gamma_{i,2t-1}, \gamma_{i,2t}}\right|
   \\ 
   \leq \  
   k_{\max}^{m} \sum_{s}  \sum_{d\geq s-1} \sum_{\{s_i\}_{i=k_{\min}}^{k_{\max}}} \sum_{\{S_i\}_{i=k_{\min}}^{k_{\max}}}\Bigg[ c^{g+m} \left( \frac{1}{N} \right)^d \left( \frac{(2\ell m)^2}{N} \right)^{(d-2g-(\ell+2)m)_+} \\  
   \ \ \ \ \ \ \ \ \ \ \ \ \ \ \cdot (4\ell m )^{6mg+6m} \prod_{i =k_{\min}}^{k_{\max}} (i (i-1) n_i)^{s_i} (k_{\max}-1)^{2g-1}\Bigg]
   \\ 
   = \  k_{\max}^m  N  \sum_{s}  \sum_{d\geq s-1}\Bigg[N^{-s}\left(\sum_{i=k_{\min}}^{k_{\max}} i(i-1)n_i\right)^s \cdot  (4\ell m )^{6mg+6m} N^{-g} c^{g+m} \\ 
   \ \ \ \ \ \ \cdot \left( \frac{(2\ell m)^2}{N} \right)^{(d-2g-(\ell+2)m)_+} (k_{\max}-1)^{2g-1}\Bigg] \\ =   k_{\max}^m  N  \sum_{s}  \sum_{d\geq s-1}(a+o(1))^s (4\ell m )^{6mg+6m} c^{g+m} \left( \frac{(2\ell m)^2}{N} \right)^{(d-2g-(\ell+2)m)_+} (k_{\max}-1)^{2g-1}   N^{-g}
\end{multline*}
By performing a change of variable $d=s+g-1$, for some constant $C>0$ we have: 
\begin{align*}
  ( \E  \| \underline{B}^{(\ell)}\|)^{2m}\leq N (C)^m \sum_{s=1}^{\infty}\sum_{g=1}^{\infty} (a+o(1))^s  (4\ell m )^{6mg+6m}  \left( \frac{(2\ell m)^2}{N} \right)^{(s-g-(\ell+2)m-1)_+} \left(\frac{C}{N}\right)^{g} \\= S_1+S_2+S_3,  
\end{align*}
where $S_1$ is the sum over $\{1
\leq s \leq \ell+2, g\geq 0 \}$, and $S_2$ is the sum taken over $\{s\geq \ell+3, g\leq  s-(\ell+2)m-1)\}$ and $S_3$ is the sum taken over the remaining terms. 

For $S_1$ we have that for some new constant $C>0$,
\begin{align}\label{bound_on_S_1}
    S_1
     \ \leq \ 
    N (a+o(1))^{\ell m} (C4\ell m )^{6m} \sum_{g=1}^{\infty } \left(\frac{(C\ell m)^{6m}}{N}\right)^{g} 
\end{align}
For sufficiently large~$n$ and by our choice of $m$, we have $\frac{(C\ell m)^{6m}}{N} \leq n^{-\frac{1}{13}}$. So 
 the series in the right hand-side of \eqref{bound_on_S_1} converges. Similar bounds can be proven for $S_2$ and $S_3$ as in the proof of \cite[Prop~14]{bordenave2015new}. We conclude that for some constant $C>0$, 
\begin{align}\label{bound_on_B_^(l)}
      \E  \| \underline{B}^{(\ell)}\|  
      \ \leq \ 
      \ell^{C}   (a+o(1))^{\frac{\ell}{2}}. 
\end{align}
Moreover for any $k \leq \ell$ one can prove $\E \| \underline{B}^{(\ell)}\|^k\leq k^C (a+o(1))^{\frac{k}{2}} $ similarly to \eqref{bound_on_B_^(l)}. Thus we conclude that
\begin{align}\label{bound on Bx}
 \E   \|\underline{B}^{(k)}\chi\| 
       \ \leq \ 
\E \|\underline{B}^{(k)}\| \sqrt{N}
      \ \leq \ 
(a+o(1))^{\frac{\ell}{2}} \sqrt{N} \ell^C.
\end{align}

Lastly, one similarly bounds $\frac{1}{N}\sum_{k=1}^{\ell} \|R^{(\ell)}_k\|$ as in the proof of \cite[Prop~18]{bordenave2015new}, with similar adjustments to those made in the proof of \eqref{bound_on_B_^(l)}. Specifically for some constant $C>0$,
\begin{align}\label{bound_on_R_k}
    \E \frac{1}{N}\sum_{k=1}^{\ell} \|R^{(\ell)}_k\| 
          \ \leq \ 
\ell^C (a+o(1))^{\frac{\ell}{2}}.
\end{align}

\subsection{Proof of Proposition~\ref{prop:second_eigenvalue}}
    By Proposition~\ref{concusion_for_simple}.\ref{tanglefree simple} and \eqref{matrix_identity_for_B^(l)} we have:
    \begin{multline*}
       \sup_{\|x\|=1:\langle (B^{*}_{G_n})^\ell\Tilde{\psi},x \rangle=0}|B^{\ell}_{G_n}  x| 
       \\
       \leq \ \  \|\underline{B}^{(\ell)}_{G_n}\|+ \frac{1}{N}\sum_{k=1}^{\ell} \|\underline{B}^{(k-1)}_{G_n}\chi\| \sup_{\|x\|=1:\langle (B^{*}_{G_n})^\ell\Tilde{\psi},x \rangle=0}|\langle\Tilde{\psi}^*, B^{(\ell-k)}_{G_n}x\rangle | +  \frac{1}{N}\sum_{k=1}^{\ell}\|R_k^{(\ell)}\| 
    \end{multline*}
    By Markov's inequality, \eqref{bound_on_B_^(l)}, \eqref{bound_on_R_k} and Theorem~\ref{thm:graphic,from multi to simple} there exists a large constant $C>0$ such that with high probability for all $k\in [\ell]$:
        \begin{align*}
      \|\underline{B}^{(\ell)}_{G_n}\|+   \frac{1}{N}\sum_{k=1}^{\ell}\|R_k^{(\ell)} \| \ \leq \ 
      \ell^C (a+o(1))^{\frac{\ell}{2}}   
      \\
      \|\underline{B}^{(k-1)}_{G_n}\chi\| \leq \ell^C (a+o(1))^{\frac{\ell}{2}}  \sqrt{N}
  \end{align*}
 By Proposition~\ref{concusion_for_simple}, Theorem~\ref{thm:graphic,from multi to simple}
 and \eqref{B^lx,}, 
  with high probability  for all $k \in [\ell]$:
  \begin{align*}
      \sup_{\|x\|=1:\langle (B^{*}_{G_n})^\ell\Tilde{\psi},x \rangle=0}|\langle\Tilde{\psi}^*, B^{(\ell-k)}_{G_n}x\rangle|=  \sup_{\|x\|=1:\langle (B^{*}_{G_n})^\ell\Tilde{\psi},x \rangle=0} | \langle (B^{*}_{G_n})^{\ell-k}\Tilde{\psi},x \rangle|\leq  a^{\frac{\ell-k}{2}} \sqrt{n}\log^Cn 
  \end{align*}
  The proof follows  from the results above.
\qed
\subsection{Open Problems}
Next we present some open problems that come naturally from our approach.
\begin{problem}
    In Equation~\eqref{defn:random_graphs_model}, can we improve the rate of convergence so that the rate would be bellow polynomial?
\end{problem}

\begin{problem}
    This paper estimates only the leading eigenvalue, can similar results be obtained for all the eigenvalues? For example, is there a limiting distribution for the empirical spectral distribution of $B$?
\end{problem}

\begin{problem}
    In this paper we assumed that the degrees of the graph are bounded from above, what if we remove this assumption (but we require finite expectation and variance)?
\end{problem}

\begin{problem}
    What if we allow vertices of degree 1, but condition on the event that the unimodular Galton Watson process does not extinct? 
\end{problem}

\begin{problem}
    Our theorem states that for every $\epsilon > 0$ and $\alpha \in (1, 2r-1)$, there exists a large enough $n$ so that there is a graph $G$ with $n$ vertices, whose universal cover has a growth rate $\alpha \pm \epsilon$.
    What if we fix $n$, and wish to understand the structure of the graph whose universal cover is closest to $\alpha$, among all graphs with $n$ vertices?
\end{problem}

\bibliographystyle{abbrv} 
\bibliography{bib}

@article{coulon2026growth,
  title={On the growth spectrum of hyperbolic groups},
  author={Coulon, R{\'e}mi and Louvaris, Michail and Wise, Daniel T and Yehuda, Gal},
  journal={arXiv preprint arXiv:2607.02147},
  year={2026}
}

@article{louvaris2025subgroups,
  title={Subgroups of a free group with every growth rate},
  author={Louvaris, Michail and Wise, Daniel T and Yehuda, Gal},
  journal={arXiv preprint arXiv:2505.10650},
  year={2025}
}

@article{hashimoto1989zeta,
  title={Zeta functions of finite graphs and representations of p-adic groups},
  author={Hashimoto, Ki-ichiro},
  journal={Advances in Studies in Pure Mathematics},
  year={1989}
}

@book{terras2011zeta,
  title={Zeta Functions of Graphs: A Stroll through the Garden},
  author={Terras, Audrey},
  year={2011},
  publisher={Cambridge University Press}
}

@article{krzakala2013spectral,
  title={Spectral redemption in clustering sparse networks},
  author={Krzakala, Florent and Moore, Cristopher and Mossel, Elchanan and Neeman, Joe and Sly, Allan and Zdeborov{\'a}, Lenka and Zhang, Pan},
  journal={PNAS},
  year={2013}
}

@article{karrer2014percolation,
  title={Percolation on sparse networks},
  author={Karrer, Brian and Newman, M. E. J.},
  journal={Physical Review E},
  year={2014}
}

@book{bollobas2001random,
  title={Random Graphs},
  author={Bollob{\'a}s, B{\'e}la},
  publisher={Cambridge University Press},
  year={2001}
}

@book{van2016random,
  title={Random Graphs and Complex Networks},
  author={van der Hofstad, Remco},
  year={2016},
  publisher={Cambridge University Press}
}

@article{bordenave2016lecture,
  title={Lecture notes on random graphs and probabilistic combinatorial optimization
!!draft in construction!!},
  journal={in preparation},
  author={Bordenave, Charles},
  year={2016}
}

@article{gulikers2016non,
  title={Non-backtracking spectrum of degree-corrected stochastic block models},
  author={Gulikers, Lennart and Lelarge, Marc and Massouli{\'e}, Laurent},
  journal={arXiv preprint arXiv:1609.02487},
  year={2016}
}

@inproceedings{bordenave2015non,
  title={Non-backtracking spectrum of random graphs: community detection and non-regular ramanujan graphs},
  author={Bordenave, Charles and Lelarge, Marc and Massouli{\'e}, Laurent},
  booktitle={2015 IEEE 56th Annual Symposium on Foundations of Computer Science},
  pages={1347--1357},
  year={2015},
  organization={IEEE}
}

@book{vershynin2018high,
  title={High-dimensional probability: An introduction with applications in data science},
  author={Vershynin, Roman},
  volume={47},
  year={2018},
  publisher={Cambridge university press}
}

@article{janson2009probability,
  title={The probability that a random multigraph is simple},
  author={Janson, Svante},
  journal={Combinatorics, Probability and Computing},
  volume={18},
  number={1-2},
  pages={205--225},
  year={2009},
  publisher={Cambridge University Press}
}

@article{bordenave2015new,
  title={A new proof of {F}riedman's second eigenvalue Theorem and its extension to random lifts},
  author={Bordenave, Charles},
  journal={arXiv preprint arXiv:1502.04482},
  year={2015}
}

@article{stephan2022non,
  title={Non-backtracking spectra of weighted inhomogeneous random graphs},
  author={Stephan, Ludovic and Massouli{\'e}, Laurent},
  journal={Mathematical Statistics and Learning},
  volume={5},
  number={3},
  pages={201--271},
  year={2022}
}

@article{benaych2020spectral,
  title={Spectral radii of sparse random matrices},
  author={Benaych-Georges, Florent and Bordenave, Charles and Knowles, Antti},
  year={2020}
}

@article{bordenave2023detection,
  title={Detection thresholds in very sparse matrix completion},
  author={Bordenave, Charles and Coste, Simon and Nadakuditi, Raj Rao},
  journal={Foundations of Computational Mathematics},
  volume={23},
  number={5},
  pages={1619--1743},
  year={2023},
  publisher={Springer}
}

@article{dumitriu2022extreme,
  title={Extreme singular values of inhomogeneous sparse random rectangular matrices},
  author={Dumitriu, Ioana and Zhu, Yizhe},
  journal={arXiv preprint arXiv:2209.12271},
  year={2022}
}

\end{document}